\documentclass[11pt]{article}

\usepackage[hyperfootnotes=false]{hyperref}
\usepackage[margin=1in]{geometry}
\usepackage{graphicx}
\usepackage{needspace}

\usepackage{amsmath}
\usepackage{amssymb}
\usepackage{xcolor}
\usepackage{bm}

\usepackage[amsmath,thmmarks]{ntheorem}
\newtheorem{theorem}{Theorem}
\newtheorem{lemma}[theorem]{Lemma}

\newtheorem{corollary}[theorem]{Corollary}
\newtheorem{definition}[theorem]{Definition}
{
	\theoremstyle{nonumberplain}
	\theoremheaderfont{\bfseries}
	\theorembodyfont{\normalfont}
	\newtheorem{remark}{Remark.}
}
{
	\theoremstyle{nonumberplain}
	\theoremheaderfont{\bfseries}
	\theorembodyfont{\normalfont}
	\theoremsymbol{\mbox{$\Box$}}
	\newtheorem{proof}{Proof.}
}

\newcommand{\abs}[1]{\left\lvert#1\right\rvert}
\newcommand{\norm}[1]{\left\lVert#1\right\rVert}
\newcommand{\R}{\mathbb R}
\newcommand{\Nzero}{\mathbb N_0}
\renewcommand{\P}{\mathbb P}
\newcommand{\I}{\mathrm{I}}
\newcommand{\II}{\mathrm{II}}
\newcommand{\III}{\mathrm{III}}
\newcommand{\fnn}{\mathrm{NN}}
\newcommand{\nslot}{\operatorname{nslot}}

\numberwithin{equation}{section}

\usepackage{biblatex}
\title{Exact Representation of Piecewise Polynomial Finite Element Functions by Two-Hidden-Layer Multilayer Perceptrons}
\author{Jun Hu\footnote{School of Mathematical Sciences, Peking University, Beijing 100871, China, and Chongqing Research Institute of Big Data, Peking University, Chongqing 401329, China (hujun@math.pku.edu.cn).} , Pengzhan Jin\footnote{National Engineering Laboratory for Big Data Analysis and Applications, Peking University, Beijing 100871, China, and Chongqing Research Institute of Big Data, Peking University, Chongqing 401329, China (jpz@pku.edu.cn).} , Shengyang Xu\footnote{School of Mathematical Sciences, Peking University, Beijing 100871, China (xushengyang0429@stu.pku.edu.cn).}}
\date{}

\begin{document}

\maketitle

\begin{abstract}
	We construct exact representations of finite element functions of arbitrary polynomial degree on simplicial meshes in any spatial dimension by two-hidden-layer multilayer perceptrons (MLPs). The first hidden layer uses $\mathrm{ReLU}^0+\mathrm{ReLU}$, and the second uses the degree-dependent activation $\mathrm{ReLU}^k$. The neural networks can realize both zero-skeleton and prescribed-skeleton representatives of discontinuous finite element functions, where the former vanish on the mesh skeleton and the latter admit independently prescribed polynomial data on each relatively open lower-dimensional mesh subsimplex; in particular, continuous finite element functions can be represented pointwise on the closed domain. In the continuous case, a half-open decomposition reduces the second-hidden-layer width by assigning each subsimplex to a single incident element. The number of nonzero parameters grows linearly with the number of elements for fixed spatial dimension, polynomial degree, and output dimension, and it is of the same order as the number of global degrees of freedom of the finite element space in the continuous case, provided that the mesh is shape-regular. All parameters in the neural network can be explicitly computed via the information of the mesh and finite element function without training.
\end{abstract}

\section{Introduction}

Neural networks are used in numerical methods for partial differential equations to parameterize trial functions, approximate solution operators, and complement classical discretizations. Their approximation properties have therefore attracted substantial attention, from universal approximation \cite{cybenko1989approximation} to quantitative relations between accuracy and network complexity; see DeVore, Hanin, and Petrova \cite{devore2021approximation} for a comprehensive survey. Finite element functions provide a concrete setting in which to study these relations: their local polynomial structure is explicit, and their approximation properties are well understood. Exact neural representations of these functions connect the two approximation frameworks and make it possible to transfer finite element error estimates to network classes.

Existing results establish this connection for several activation functions and finite element spaces. Arora et al. \cite{arora2018understanding} represented continuous piecewise linear functions by ReLU networks with depth logarithmic in the spatial dimension, and He et al. \cite{he2020relu} studied ReLU representations of linear finite elements and their depth requirements. A two-hidden-layer ReLU construction is available on a two-dimensional uniform mesh \cite{he2022hierarchical}. Jin \cite{jin2026twohidden} obtained two-hidden-layer weak representations of continuous and discontinuous piecewise linear functions on convex polytope meshes, with exact agreement away from arbitrarily thin neighborhoods of the mesh skeleton.

Higher-order approximation and exact polynomial representation provide complementary extensions of this connection. Opschoor, Petersen, and Schwab \cite{opschoor2020highorder} studied deep ReLU approximations of high-order finite element functions and the transfer of their approximation rates to neural networks. For exact polynomial realization, Li, Tang, and Yu \cite{li2020better,li2020powernet} developed constructive representations using rectified power units (RePU), together with approximation estimates for smooth functions. He, Mao, and Xu \cite{he2023expressivity} further studied polynomial expressivity and approximation by ReLU$^k$ networks, including quantitative bounds on network complexity and weights.

For arbitrary polynomial degree and spatial dimension, He and Xu \cite{he2023deep} constructed deeper networks using ReLU and squared-ReLU activations. Opschoor and Schwab \cite{opschoor2024exponential} gave exact emulations of high-order continuous finite element spaces using the same activations, with the number of nonzero parameters proportional to the finite element degrees of freedom on shape-regular meshes.

Discontinuous activations also provide a way to localize functions to mesh elements. Longo, Opschoor, Disch, Schwab, and Zech \cite{longo2023derham} used ReLU and ReLU$^0$ activations to emulate lowest-order spaces in the discrete de Rham complex on general regular simplicial meshes. Their higher-order extension uses ReLU, rectified-power, and ReLU$^0$ activations, gives exact agreement on open elements and hence almost everywhere, and has depth of order $d\log(k+1)$ and size linear in the number of elements for fixed degree and dimension. Their results include compatible vector-valued spaces such as Raviart--Thomas and N\'{e}d\'{e}lec elements.

The present paper combines three properties in a single construction: a uniform exact two-hidden-layer architecture for arbitrary dimension and polynomial degree, sparse storage linear in the number of elements, and parameters computed directly from finite element data. The representation applies to arbitrary conforming simplicial meshes; no uniform mesh structure is required. The number of hidden layers remains two, while the activation in the second layer depends on the polynomial degree. Our contributions are as follows.

\begin{enumerate}
\item \emph{A uniform two-hidden-layer representation.}
The first hidden layer encodes mesh geometry, the second combines localization with polynomial reproduction, and the linear output layer assembles the function. The MLP uses $\mathrm{ReLU}^0+\mathrm{ReLU}$ in its first hidden layer and $\mathrm{ReLU}^k$ in its second, with $\mathrm{ReLU}^0(t)=\mathbf1_{\{t>0\}}$. For the construction, we separate the first activation into two channels and then recover an ordinary MLP by an explicit parameter transformation. All components of a vector- or tensor-valued function share the same hidden features. The resulting networks represent discontinuous finite element functions almost everywhere and continuous finite element functions at every point of the closed domain. More generally, they allow independent polynomial data on all relatively open mesh subsimplices.

\item \emph{Linear sparse storage at fixed depth.}
For fixed spatial dimension, polynomial degree, and output dimension, both hidden-layer widths and sparse storage grow at most linearly with the number of elements. Sharing first-layer features across coplanar facets ties the first-layer width to the number of distinct facet-supporting hyperplanes. For continuous functions, a half-open decomposition avoids separate expansions on lower-dimensional subsimplices and gives the same second-layer width as the zero-skeleton discontinuous construction. In the scalar continuous case with $k\geqslant1$, the retained storage count is comparable to the dimension of finite element space on shape-regular mesh (Corollary~\ref{cor:dof optimality}).

\item \emph{Explicit parameter computation.}
Once the mesh, degree, and geometric assignments are fixed, all hidden weights and biases are independent of the target function. The finite element data enter linearly through the output coefficients. On each element, these coefficients can be recovered from Lagrange nodal values using a matrix that depends only on dimension and degree. Its factorization is reusable across elements and components, so the network can be generated without training or nonlinear parameter fitting. Section~\ref{sec:polynomial representation} gives the coefficient system, and Appendix~\ref{app:coefficient computation} provides a Newton-basis factorization for its evaluation.
\end{enumerate}

The role of skeleton values makes the representation bounds more precise. Let $N_\mathcal T$ be the number of elements, $N_\mathcal S$ the number of all relatively open nonempty mesh subsimplices, and $H_\mathcal T$ the number of distinct unoriented hyperplanes supporting mesh facets; coplanar facets contribute only once. Set $M=\binom{d+k}{d}$. The actual MLP hidden-layer widths are at most $(4H_\mathcal T,MN_\mathcal T)$ for the representative that equals the given element polynomial in each open element and vanishes on the skeleton (Theorem~\ref{thm: fnn interior} and Corollary~\ref{cor: dg multi}). Prescribing polynomial values independently on every subsimplex gives widths $(4H_\mathcal T,MN_\mathcal S)$ (Theorem~\ref{thm:subsimplex representation}); this includes pointwise continuous representation on arbitrary meshes. If the mesh admits the half-open decomposition of Definition~\ref{def:half open decomposition}, continuous functions have a representation with widths $(4H_\mathcal T,MN_\mathcal T)$ (Theorem~\ref{thm:continuous FE representation}). The boundary-aware star condition supplies one sufficient condition for this reduction (Corollary~\ref{cor:continuous star representation}). The two-channel notation used in the proofs counts shared affine preactivations, so its descriptor $\fnn_d^m(k;n_1,n_2)$ corresponds to actual widths $(2n_1,n_2)$.

The vector- and tensor-valued construction applies in particular to Raviart--Thomas, N\'{e}d\'{e}lec, and Hu--Zhang spaces through their inclusion in suitable polynomial DG spaces. Exact agreement almost everywhere preserves their Sobolev classes and interelement traces; assigned pointwise skeleton values are separate data. Exact realization also transfers finite element approximation estimates to neural networks for $1\leqslant p\leqslant\infty$: Corollary~\ref{cor:DG Sobolev approximation} gives broken $W^{s,p}$ estimates on shape-regular meshes, while Corollary~\ref{cor:continuous approximation} gives global $W^{s,p}$ estimates for $s=0,1$ using continuous network approximants on uniform Kuhn meshes. The latter is used to obtain explicit width estimates for the approximation result.

Although the width estimates are constructive upper bounds, two hidden layers are necessary in the worst case in the stated activation model. Indeed, for $d\geqslant2$, a finite-width one-hidden-layer network mixing ReLU$^0$ and ReLU$^k$ activations is, on any given hyperplane, either differentiable almost everywhere or nondifferentiable almost everywhere with respect to $(d-1)$-dimensional surface measure. A continuous piecewise affine finite element function, however, may have both differentiable and nondifferentiable subsets of positive surface measure on the same hyperplane, as illustrated in Figure~\ref{fig:depth counterexample}. Such a function cannot be represented by a one-hidden-layer network, establishing the optimality of the two-hidden-layer depth.

\begin{figure}[htbp]
    \centering
    \includegraphics[width=.5\textwidth]{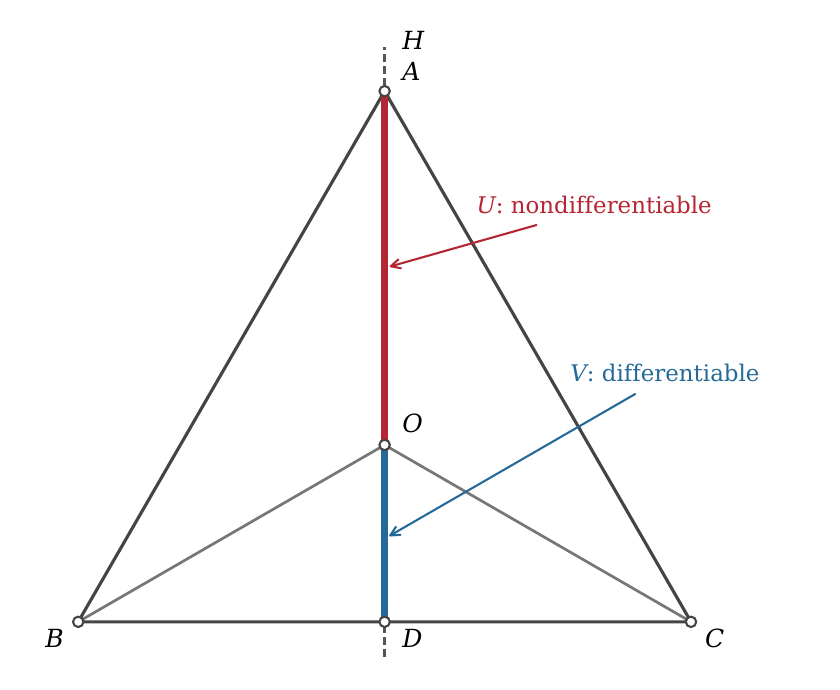}
    \caption{An equilateral triangle $ABC$ split into three subtriangles by joining its center $O$ to the vertices. Let $\phi$ be the continuous piecewise linear nodal function with $\phi(O)=1$ and $\phi(A)=\phi(B)=\phi(C)=0$. The line $H$ through $A$ and $O$ meets $BC$ at $D$. The function $\phi$ is nondifferentiable on $(A,O)$ and differentiable on $(O,D)$. Both segments have positive length.}
    \label{fig:depth counterexample}
\end{figure}

The remainder of the paper is organized as follows. Section~2 introduces the meshes, finite element spaces, network model, and local polynomial basis. Section~3 develops the zero-skeleton and prescribed-skeleton representations, including vector- and tensor-valued examples and broken Sobolev approximation estimates. Section~4 derives continuous representations, sparse-storage bounds, and continuous Sobolev approximation estimates. Section~\ref{sec:conclusion} concludes the paper, and Appendix~\ref{app:coefficient computation} details the local coefficient computation.

\section{Preliminaries}

Throughout the paper, $d\geqslant1$ denotes the spatial dimension and $k\in\Nzero$ the polynomial degree, where $\Nzero=\{0,1,2,\dots\}$. We write
$A\lesssim_{r}B$ if $A\leqslant CB$ for some constant $C>0$
depending only on $r$, and $A\gtrsim_{r}B$ if
$B\lesssim_{r}A$. We write $A\asymp_{r}B$ if both inequalities
hold. Other subscripts are interpreted analogously.

\subsection{Simplicial meshes}\label{sec:simplicial meshes}

Let $\Omega\subset\R^d$ be a bounded open polyhedral set, not necessarily convex or connected. Let $\mathcal T=\{\tau_1,\ldots,\tau_{N_\mathcal T}\}$ be a conforming simplicial mesh of $\Omega$. Each element $\tau_i\subset\Omega$ is an open $d$-simplex, and
\begin{align}
	\bigcup_{\tau\in\mathcal T}\overline\tau=\overline\Omega,
	\qquad
	\tau_i\cap\tau_j=\varnothing\quad\text{for }i\neq j.
\end{align}
Conformity means that $\overline\tau_i\cap\overline\tau_j$ is either empty or a common subsimplex of both elements. In particular, intersections of element closures introduce no hanging nodes. The mesh skeleton is $\partial\mathcal T=\bigcup_{\tau\in\mathcal T}\partial\tau$. We denote by $H_\mathcal T$ the number of distinct unoriented affine hyperplanes supporting the $(d-1)$-dimensional element facets. Coplanar facets contribute only once.

Let $\mathcal S_\mathcal T$ be the collection of all distinct relatively open nonempty subsimplices of the mesh, including the open elements themselves as well as the vertexes. Mesh conformity implies that these subsimplices form a disjoint partition of $\overline\Omega$, and those of dimension less than $d$ form a disjoint partition of the skeleton $\partial\mathcal T$. We write $N_\mathcal S=\abs{\mathcal S_\mathcal T}$ for the total number of these subsimplices. Since every $d$-simplex has $2^{d+1}-1$ nonempty subsimplices, $N_\mathcal S\leqslant(2^{d+1}-1)N_\mathcal T$.

Suppose that $\tau$ is a $d$-dimensional simplex with vertices $\bm{x}_0,\dots,\bm{x}_d$. Classical simplex theory gives the affine representation $\tau=\{\bm x:\bm W\bm x+\bm b>\bm0\}$, where $\bm W\in\R^{(d+1)\times d}$ and $\bm b\in\R^{d+1}$, and the inequality is understood componentwise. Let
\begin{align}
	\bm\lambda^\tau(\bm x)=\bm W\bm x+\bm b,
	\qquad
	\lambda_i^\tau(\bm x)=\bm w_i\cdot\bm x+b_i.
\end{align}
The equation $\lambda_i^\tau(\bm x)=0$ defines the hyperplane containing the facet opposite the vertex $\bm x_i$. We normalize $\lambda_i^\tau$ by requiring $\lambda_i^\tau(\bm x_i)=1$. The functions $\lambda_i^\tau$ are then precisely the barycentric coordinate functions and satisfy $\sum_{i=0}^d\lambda_i^\tau(\bm x)=1$ for every $\bm x\in\R^d$. In terms of barycentric coordinates, every relatively open subsimplex $\sigma$ of $\overline\tau$ is characterized by
\begin{align}\label{eq:subsimplex bcf}
	\sigma=\Big\{\bm x: \sum_{i=0}^d\lambda_i^\tau(\bm x)=1;
	\ \lambda_i^\tau(\bm x)>0\text{ for }i\in\mathcal I_{\tau,\sigma};
	\ \lambda_i^\tau(\bm x)=0\text{ for }i\notin\mathcal I_{\tau,\sigma}\Big\}.
\end{align}
Here $\mathcal I_{\tau,\sigma}\subset\{0,\dots,d\}$ is the set of indices corresponding to the vertices of $\tau$ that belong to $\overline\sigma$.
The superscript $\tau$ on $\bm\lambda^\tau$ will be omitted when no ambiguity can arise.

\subsection{Finite element spaces}\label{sec:finite element spaces}

Throughout the remainder of the paper, let $\mathbb E$ be a finite-dimensional real vector space with $m=\dim\mathbb E\geqslant1$, such as $\R$, $\R^d$, or the space of symmetric $d\times d$ matrices. Fix a basis of $\mathbb E$ and identify it with $\R^m$. We write $\P_k(\tau;\mathbb E)=\P_k(\tau)\otimes\mathbb E$ for the $\mathbb E$-valued polynomials of total degree at most $k$. The scalar case $\mathbb E=\R$ is understood whenever the value space is omitted from the notation.

The usual discontinuous finite element space, whose members are identified up to almost-everywhere equality, is
\begin{align}
	\mathcal V_k^{\mathrm{DG}}(\mathcal T;\mathbb E)
	=
	\left\{\bm v\in L^2(\Omega;\mathbb E):\bm v|_\tau\in\P_k(\tau;\mathbb E)
	\text{ for every }\tau\in\mathcal T\right\}.
\end{align}
These equivalence classes do not prescribe values on the mesh skeleton, whereas a neural network is defined pointwise. We therefore introduce the space
\begin{align}
	\widehat{\mathcal V}_k^{\mathrm{DG}}(\mathcal T;\mathbb E)
	=
	\left\{\bm v:\overline\Omega\to\mathbb E:
	\bm v|_\sigma\in\P_k(\R^d;\mathbb E)|_\sigma
	\text{ for every }\sigma\in\mathcal S_\mathcal T\right\}.
\end{align}
No compatibility is imposed between the polynomials on distinct relatively open subsimplices.

The zero-skeleton representatives form the space
\begin{align}
	\mathring{\mathcal V}_k^{\mathrm{DG}}(\mathcal T;\mathbb E)
	=
	\left\{\bm v:\overline\Omega\to\mathbb E:
	\begin{array}{l}
	\bm v|_\tau\in\P_k(\tau;\mathbb E)\text{ for every }\tau\in\mathcal T,\\
	\bm v|_{\partial\mathcal T}=\bm0
	\end{array}\right\}.
\end{align}
Passing to almost-everywhere equivalence classes gives a linear bijection from $\mathring{\mathcal V}_k^{\mathrm{DG}}(\mathcal T;\mathbb E)$ onto $\mathcal V_k^{\mathrm{DG}}(\mathcal T;\mathbb E)$.

The continuous finite element space is
\begin{align}
	\mathcal V_k(\mathcal T;\mathbb E)
	=
	\left\{\bm v\in C(\overline\Omega;\mathbb E):\bm v|_\tau\in\P_k(\tau;\mathbb E)
	\text{ for every }\tau\in\mathcal T\right\}.
\end{align}
Continuity uniquely determines the skeleton values from the elementwise polynomials. The two pointwise spaces satisfy
\begin{align}
	\mathring{\mathcal V}_k^{\mathrm{DG}}(\mathcal T;\mathbb E),
	\quad \mathcal V_k(\mathcal T;\mathbb E)
	\subset\widehat{\mathcal V}_k^{\mathrm{DG}}(\mathcal T;\mathbb E).
\end{align}

In sums and inclusions involving $L^2$, Sobolev, or DG equivalence-class spaces, we identify pointwise functions with their almost-everywhere classes. Equality within $\widehat{\mathcal V}_k^{\mathrm{DG}}(\mathcal T;\mathbb E)$ remains pointwise.

Each of these four spaces is the tensor product of its scalar counterpart with $\mathbb E$. The $m$ components of an $\mathbb E$-valued function will correspond to the $m$ network outputs in Section~\ref{sec:network model}.

\subsection{Neural network model}\label{sec:network model}

We use the activations
\begin{align}
	\varrho_0(t)=\mathbf 1_{\{t>0\}},
	\qquad
	\varrho_j(t)=(t_+)^j\quad\text{for }j\geqslant1,
\end{align}
where $t_+=\max\{t,0\}$, and all activations are applied componentwise to vectors. Simply put, $\varrho_j(t)$ is exactly the ${\rm ReLU}^j$ for $j\geqslant 0$.
Here $\mathbf 1_E$ denotes the indicator of the set or condition $E$.

\begin{definition}[Two-hidden-layer network class]\label{def:FNN}
For positive integers $d,m,n_1,n_2$ and $k\in\Nzero$, let $\fnn_d^m(k;n_1,n_2)$ denote the class of functions $\bm{\mathcal N}:\R^d\to\R^m$ represented by the two-hidden-layer architecture
\begin{align}\label{eq:FNN architecture}
    \bm{\mathcal N}(\bm x)=\bm W_\III\varrho_k\!\left(
    \bm W_{\II,0}\varrho_0(\bm W_\I\bm x+\bm b_\I)
    +\bm W_{\II,1}\varrho_1(\bm W_\I\bm x+\bm b_\I)
    +\bm b_\II\right),
\end{align}
where $\bm W_\I\in\R^{n_1\times d}$, $\bm b_\I\in\R^{n_1}$, $\bm W_{\II,0},\bm W_{\II,1}\in\R^{n_2\times n_1}$, $\bm b_\II\in\R^{n_2}$, and $\bm W_\III\in\R^{m\times n_2}$. For a function class $\mathcal V$ on a domain $D$, we say that $\fnn_d^m(k;n_1,n_2)$ realizes $\mathcal V$ pointwise on $D$ if every $v\in\mathcal V$ agrees on $D$ with some $\bm{\mathcal N}\in\fnn_d^m(k;n_1,n_2)$.
\end{definition}

To relate this representation to ordinary MLPs, set $\varrho_\ast=\varrho_0+\varrho_1$ and define
\begin{align}\label{eq:uniform activation MLP}
    \mathrm{MLP}_d^m(k;r_1,r_2)
    =\left\{\bm x\mapsto\widetilde{\bm W}_\III\varrho_k\!\left(
    \widetilde{\bm W}_\II\varrho_\ast(\widetilde{\bm W}_\I\bm x+\widetilde{\bm b}_\I)+\widetilde{\bm b}_\II\right)\right\},
\end{align}
where $\widetilde{\bm W}_\I\in\R^{r_1\times d}$, $\widetilde{\bm b}_\I\in\R^{r_1}$, $\widetilde{\bm W}_\II\in\R^{r_2\times r_1}$, $\widetilde{\bm b}_\II\in\R^{r_2}$, and $\widetilde{\bm W}_\III\in\R^{m\times r_2}$ vary freely. The activations are $\varrho_\ast$ and $\varrho_k$ in the first and second hidden layers, respectively, and $(r_1,r_2)$ are their actual widths.

For every $t\in\R$, including $t=0$,
\begin{align}
    \varrho_0(t)=2\varrho_\ast(t)-\varrho_\ast(2t),
    \qquad \varrho_1(t)=\varrho_\ast(2t)-\varrho_\ast(t).
\end{align}
Thus \eqref{eq:FNN architecture} is realized by an MLP with
\begin{gather}
    \widetilde{\bm W}_\I=\begin{pmatrix}\bm W_\I\\2\bm W_\I\end{pmatrix},\quad 
    \widetilde{\bm b}_\I=\begin{pmatrix}\bm b_\I\\2\bm b_\I\end{pmatrix},\notag\\
    \widetilde{\bm W}_\II=\begin{pmatrix}2\bm W_{\II,0}-\bm W_{\II,1}&
    \bm W_{\II,1}-\bm W_{\II,0}\end{pmatrix},\quad
    \widetilde{\bm b}_\II=\bm b_\II,\quad
    \widetilde{\bm W}_\III=\bm W_\III.\label{eq:MLP parameter conversion}
\end{gather}
Consequently, as classes of functions on $\R^d$,
\begin{align}\label{eq:NN MLP inclusion}
    \fnn_d^m(k;n_1,n_2)\subseteq\mathrm{MLP}_d^m(k;2n_1,n_2).
\end{align}
The construction uses paired first-layer parameters, whereas the MLP class allows unrestricted parameters. In the NN descriptor, $n_1$ counts shared affine preactivations, each supplying two activation channels. Its scalar-neuron widths, and those of the constructed MLP, are therefore $(2n_1,n_2)$. In particular, $\fnn_d^m(k;2H_\mathcal T,MN_\mathcal T)$ gives MLP widths $(4H_\mathcal T,MN_\mathcal T)$. We retain the two-channel representation in subsequent proofs. \eqref{eq:MLP parameter conversion} gives all MLP parameters explicitly and preserves sparse-storage scaling up to a constant factor.

From the MLP viewpoint, $\varrho_\ast$ is discontinuous at zero but satisfies
\begin{align}
    \varrho_\ast'(t)=\begin{cases}0,&t<0,\\1,&t>0.\end{cases}
\end{align}
For $k\geqslant1$, at inputs where all hidden preactivations are nonzero, a differentiable loss can therefore be differentiated with respect to the MLP weights and biases by the usual chain rule. Gradients can propagate through both hidden layers, although they may vanish and do not account for jumps across activation thresholds. If the paired parameterization is retained during differentiation, contributions from both members of each pair are combined by the chain rule. For $k=0$, hidden-parameter gradients vanish away from thresholds because the second-layer activation is $\varrho_0$. The representation results below require no gradient-based training: geometry and degree determine the hidden parameters, and local linear algebra determines the output coefficients.

\subsection{Polynomials on a simplex}\label{sec:polynomial representation}

Let $\tau$ be a $d$-dimensional simplex. Define the multi-index set
\begin{align}
	\mathcal A_k^d
	=
	\left\{\bm\alpha\in\Nzero^d:\abs{\bm\alpha}\leqslant k\right\}.
\end{align}
We use ${}^\circ$ for zero extension from $\Nzero^d$ to $\Nzero^{d+1}$, namely $\bm\alpha^\circ=(0,\alpha_1,\dots,\alpha_d)$.
The symbol $\bm1$ denotes the all-ones vector of the dimension determined by the context.
The polynomial space of total degree at most $k$ is
\begin{align}
	\P_k(\tau)
	=
	\operatorname{span}\left\{\bm x^{\bm\alpha}:\bm\alpha\in\mathcal A_k^d\right\}.
\end{align}

To connect $\P_k(\tau)$ with neural networks, we use affine functions that are homogeneous in the barycentric coordinates. For $\bm\beta\in\mathcal A_k^d$, define
\begin{align}
	L_{\bm\beta}^\tau(\bm x)
	&=
	(\bm1+\bm\beta^\circ)\cdot\bm\lambda^\tau(\bm x)\notag\\
	&=
	\lambda_0^\tau(\bm x)
	+
	\sum_{i=1}^d(\beta_i+1)\lambda_i^\tau(\bm x)
	=
	1+\sum_{i=1}^d\beta_i\lambda_i^\tau(\bm x).
\end{align}
Since $\bm\lambda^\tau$ consists of barycentric coordinates,
\begin{align}\label{eq:L bounds}
	1\leqslant L_{\bm\beta}^\tau(\bm x)\leqslant k+1
	\qquad\text{for every }\bm x\in\overline\tau.
\end{align}

The following basis property is a direct consequence of the classical unisolvence of simplex principal lattices; see \cite[Section~2.2]{gasca2000polynomial}.

\begin{lemma}[A polynomial basis of affine powers]\label{lem:k-linear decompose of Pk}
    The functions
    \begin{align}\label{eq:k-linear decompose of Pk}
        \left\{(L_{\bm\beta}^\tau)^k:\bm\beta\in\mathcal A_k^d\right\}
    \end{align}
    form a basis of $\P_k(\tau)$. For $k=0$, this basis consists of the constant function $1$.
\end{lemma}

\begin{proof}
The case $k=0$ is immediate. For $k\geqslant1$, let $F_\tau:\widehat\tau\to\tau$ be the affine bijection from the standard simplex, mapping its vertices $\bm0,\bm e_1,\ldots,\bm e_d$ to the ordered vertices $\bm x_0,\ldots,\bm x_d$ of $\tau$. Since $\lambda_i^\tau\circ F_\tau(\widehat{\bm x})=\widehat x_i$ for $1\leqslant i\leqslant d$, the multinomial theorem gives
\begin{align}
    \bigl((L_{\bm\beta}^\tau)^k\circ F_\tau\bigr)(\widehat{\bm x})
    =(1+\bm\beta\cdot\widehat{\bm x})^k
    =\sum_{\bm\alpha\in\mathcal A_k^d}
    \frac{k!}{(k-|\bm\alpha|)!\,\bm\alpha!}
    \bm\beta^{\bm\alpha}\widehat{\bm x}^{\bm\alpha}.
\end{align}
Thus the coefficient matrix in the monomial basis is $\bm D\bm V$, where $\bm D$ is diagonal with positive entries and $V_{\bm\alpha\bm\beta}=\bm\beta^{\bm\alpha}$. Here $V_{\bm\alpha\bm\beta}$ is understood as
$V_{\iota(\bm\alpha),\iota(\bm\beta)}$, where
$\iota:\mathcal A_k^d\to\{1,\ldots,M\}$ is a fixed
total-degree-compatible ordering and $M=\binom{d+k}{d}$. Thus $\iota(\bm\alpha)<\iota(\bm\beta)$ whenever $|\bm\alpha|<|\bm\beta|$. The nodes $\{\bm\beta:\bm\beta\in\mathcal A_k^d\}$ are a dilation of the degree-$k$ principal lattice on the standard simplex and hence are unisolvent for $\P_k(\R^d)$ \cite[Section~2.2]{gasca2000polynomial}. Therefore $\bm V$, the transpose of their evaluation matrix, is invertible. The pulled-back functions therefore form a basis of $\P_k(\widehat\tau)$. Since pullback by $F_\tau$ is an isomorphism from $\P_k(\tau)$ to $\P_k(\widehat\tau)$, the claimed basis property follows.
\end{proof}

For fixed $d$ and $k$, the coefficients in Lemma~\ref{lem:k-linear decompose of Pk} can be recovered from the local Lagrange nodal values by the same reference matrix on every element. For $k\geqslant1$, this system is
\begin{align}
    \bm A\bm u=\bm p,
    \qquad
    A_{\bm\alpha\bm\beta}
    =\left(1+\frac{\bm\alpha\cdot\bm\beta}{k}\right)^k,
    \qquad \bm\alpha,\bm\beta\in\mathcal A_k^d.
\end{align}
Here $\bm p$ contains the nodal values in the ordering used above. $\bm A$ is independent of the position, size, and shape of $\tau$, because the nodes have fixed barycentric coordinates. Its factorization can therefore be reused on every element. For $k=0$, the single coefficient is the constant value of the polynomial. Appendix~\ref{app:coefficient computation} reuses the Vandermonde matrix $\bm V$ above to factor $\bm A$ and gives an explicit Newton-basis triangular factorization of $\bm V$. This yields a coefficient computation using triangular solves, matrix-vector products, and diagonal scaling.

Lagrange nodal values are only one choice for recovering the coefficients; any unisolvent set of linear degrees of freedom on $\P_k(\tau)$ gives an invertible system. The matrix remains independent of the element geometry if these degrees of freedom are defined by pulling each polynomial back to the reference simplex and applying a fixed set of reference degrees of freedom.

\section{Pointwise realization of discontinuous finite element functions}
\subsection{Scalar zero-skeleton representation}\label{sec:zero skeleton representation}

The following lemma is the foundation of our construction:

\begin{lemma}[Localization to open simplices]\label{lem: nn interior}
	For each $\tau\in\mathcal T$ and $\bm\beta\in\mathcal A_k^d$, the second-layer preactivation
	\begin{align}\label{eq:interior localizer}
		f_{\bm\beta}^\tau(\bm x)
		=
		(k+1)\bm1\cdot\varrho_0(\bm\lambda^\tau(\bm x))
		+
		(\bm\beta^\circ-(k+1)\bm1)\cdot
		\varrho_1(\bm\lambda^\tau(\bm x))
		+
		1-(k+1)d,
	\end{align}
	formed from the two first-layer activation channels of the $d+1$ affine preactivations $\bm\lambda^\tau(\bm x)$, satisfies
	\begin{align}
		f_{\bm\beta}^\tau(\bm x)=L_{\bm\beta}^\tau(\bm x)
		\quad\text{for }\bm x\in\tau,
		\qquad
		f_{\bm\beta}^\tau(\bm x)\leqslant0
		\quad\text{for }\bm x\notin\tau.
	\end{align}
\end{lemma}

\begin{proof}
	Define $g_{\bm\beta}^\tau(\bm x)=(\bm\beta^\circ-(k+1)\bm1)\cdot\varrho_1(\bm\lambda^\tau(\bm x))$. On $\tau$, all barycentric coordinates are positive, so $g_{\bm\beta}^\tau(\bm x)=(\bm\beta^\circ-(k+1)\bm1)\cdot\bm\lambda^\tau(\bm x)=L_{\bm\beta}^\tau(\bm x)-k-2$.

	Let $\bm\lambda^+=\varrho_1(\bm\lambda^\tau)$ and $\bm\lambda^-=\varrho_1(-\bm\lambda^\tau)$. Then $\bm\lambda^\tau=\bm\lambda^+-\bm\lambda^-$ and $\bm1\cdot\bm\lambda^+=1+\bm1\cdot\bm\lambda^-\geqslant1$.
	Since every component of $\bm\beta^\circ-k\bm1$ is nonpositive,
	\begin{align}\label{eq:g global bound}
		g_{\bm\beta}^\tau(\bm x)
		&=
		\bm\lambda^+\cdot(\bm\beta^\circ-k\bm1)
		-
		\bm1\cdot\bm\lambda^+
		\leqslant-1
	\end{align}
	for every $\bm x\in\R^d$.

	Next define $h^\tau(\bm x)=(k+1)\bm1\cdot\varrho_0(\bm\lambda^\tau(\bm x))+1-(k+1)d$.
	On $\tau$, all $d+1$ ReLU$^0$ functions equal one and therefore $h^\tau(\bm x)=k+2$. If $\bm x\in\tau^c$, at least one barycentric coordinate is nonpositive. Since $\varrho_0(0)=0$, at most $d$ of the ReLU$^0$ functions equal one, so $h^\tau(\bm x)\leqslant1$.
	Finally, $f_{\bm\beta}^\tau=g_{\bm\beta}^\tau+h^\tau$. On $\tau$ this gives $f_{\bm\beta}^\tau=L_{\bm\beta}^\tau$, while on $\tau^c$ equation~\eqref{eq:g global bound} and the bound on $h^\tau$ give $f_{\bm\beta}^\tau\leqslant0$.
\end{proof}

The same argument applies to other subsets of $\overline\tau$. Suppose that $R\subset\overline\tau$ is described by $d+1$ binary tests $\chi_i$, so that $\bm x\in R$ if and only if all $\chi_i(\bm x)$ equal one. Then
\begin{align}\label{eq:indicator localizer}
    (\bm\beta^\circ-(k+1)\bm1)\cdot\varrho_1(\bm\lambda^\tau(\bm x))
    +(k+1)\sum_{i=0}^d{\chi_i^\tau}(\bm x)+1-(k+1)d
\end{align}
equals $L_{\bm\beta}^\tau$ on $R$ and is nonpositive outside $R$. Indeed, the ReLU term equals $L_{\bm\beta}^\tau-k-2$ on $R$ and is at most $-1$ everywhere by \eqref{eq:g global bound}. The sum of the tests is $d+1$ on $R$ and at most $d$ elsewhere. This formula will be used in the proofs of Lemmas~\ref{lem:single subsimplex localizer} and~\ref{lem:assigned element localizer} with suitable choices of the tests ${\chi_i^\tau}$.

The construction above is local to each simplex. In particular, it does not require $\Omega$ to be convex, connected, or star-shaped.

\begin{theorem}[Scalar representation]\label{thm: fnn interior}
	Let $M=\binom{d+k}{d}=\dim\P_k(\R^d)$. Then ${\fnn}_d^1(k;2H_\mathcal T,MN_\mathcal T)$ realizes $\mathring{\mathcal V}_k^{\mathrm{DG}}(\mathcal T)$ pointwise on $\overline\Omega$. For fixed $d$ and $k$, all hidden parameters are determined by the mesh, and only the output coefficients vary with $v$.
\end{theorem}

In particular, the theorem includes $k=0$, in which case $M=1$ and the second hidden layer uses $\varrho_0$.

\begin{proof}
	Let $v\in\mathring{\mathcal V}_k^{\mathrm{DG}}(\mathcal T)$. For every $\tau\in\mathcal T$, Lemma~\ref{lem:k-linear decompose of Pk} gives coefficients $u_{\bm\beta}^\tau$ such that
	\begin{align}
		v|_\tau
		=
		\sum_{\bm\beta\in\mathcal A_k^d}
		u_{\bm\beta}^\tau(L_{\bm\beta}^\tau)^k.
	\end{align}
	By Lemma~\ref{lem: nn interior},
	\begin{align}\label{eq:zero skeleton representation}
		v(\bm x)
		=
		\sum_{\tau\in\mathcal T}
		\sum_{\bm\beta\in\mathcal A_k^d}
		u_{\bm\beta}^\tau
		\varrho_k(f_{\bm\beta}^\tau(\bm x)).
	\end{align}
	Indeed, if $\bm x\in\tau$, the summands associated with $\tau$ reproduce $v|_\tau$, whereas all summands associated with other elements vanish. If $\bm x\in\partial\mathcal T$, every localizer is nonpositive, so the right-hand side vanishes, as required. This argument also applies when $k=0$ because $\varrho_0(t)=0$ for every $t\leqslant0$.

	To write \eqref{eq:zero skeleton representation} in network form, stack the elementwise affine maps as
	\begin{align}
		\bm W_\I
		=
		\begin{pmatrix}
		\bm W_\I^{\tau_1}\\
		\bm W_\I^{\tau_2}\\
		\vdots\\
		\bm W_\I^{\tau_{N_\mathcal T}}
		\end{pmatrix},
		\qquad
		\bm b_\I
		=
		\begin{pmatrix}
		\bm b_\I^{\tau_1}\\
		\bm b_\I^{\tau_2}\\
		\vdots\\
		\bm b_\I^{\tau_{N_\mathcal T}}
		\end{pmatrix}.
	\end{align}
	Here $\bm{W}_\I^\tau \bm{x} + \bm{b}_\I^\tau = \bm{\lambda}^\tau(\bm x)$. For $i\in\{0,1\}$, let
	\begin{align}
		\bm W_{\II,i}
		=
		\begin{pmatrix}
		\bm W_{\II,i}^{\tau_1}&&&\\
		&\bm W_{\II,i}^{\tau_2}&&\\
		&&\ddots&\\
		&&&\bm W_{\II,i}^{\tau_{N_\mathcal T}}
		\end{pmatrix},
	\end{align}
	where the rows of $\bm W_{\II,0}^\tau$ and $\bm W_{\II,1}^\tau$ are, respectively, the coefficients of the $\varrho_0$ and $\varrho_1$ channels in \eqref{eq:interior localizer}, ordered by $\bm\beta\in\mathcal A_k^d$. Stack the constant terms from \eqref{eq:interior localizer} into $\bm b_\II$, and define
	\begin{align}
		\bm W_\III
		=
		\begin{pmatrix}
		\bm W_\III^{\tau_1}&
		\bm W_\III^{\tau_2}&
		\cdots&
		\bm W_\III^{\tau_{N_\mathcal T}}
		\end{pmatrix},
		\qquad
		\bm W_\III^\tau
		=
		\begin{pmatrix}
		u_{\bm\beta_1}^\tau&\cdots&u_{\bm\beta_M}^\tau
		\end{pmatrix}.
	\end{align}
	For fixed $d$ and $k$, the matrices $\bm W_\I$, $\bm W_{\II,i}$, $i=0,1$, and the biases $\bm b_\I$, $\bm b_\II$ are determined by the mesh. Only $\bm W_\III$ changes with $v$.

	Before removing repetitions, $\bm W_\I\in\R^{N_\mathcal T(d+1)\times d}$ and $\bm W_{\II,i}\in\R^{MN_\mathcal T\times N_\mathcal T(d+1)}$, and
	\begin{align}
		v(\bm x)
		=
		\bm W_\III\varrho_k\!\left(
		\bm W_{\II,0}\varrho_0(\bm W_\I\bm x+\bm b_\I)
		+
		\bm W_{\II,1}\varrho_1(\bm W_\I\bm x+\bm b_\I)
		+
		\bm b_\II
		\right).
	\end{align}
    The second hidden layer has width $MN_\mathcal T$. For each unoriented mesh hyperplane, include both affine orientations in the first-layer preactivation vector. Every barycentric coordinate used by an element is a positive multiple of one of these oriented affine functions. Repeated functions can be merged, with the positive scaling absorbed into the corresponding second-layer weights. Hence at most $2H_\mathcal T$ distinct affine preactivations are required.
\end{proof}

\begin{remark}[Function-dependent width reduction]
    The widths in Theorem~\ref{thm: fnn interior} are sufficient for the stated representation; they are not asserted to be minimal. The construction uses hidden parameters prescribed by the mesh, with only the output weights depending on the target function. Allowing the hidden parameters to depend on that function can reduce the required width, as the following example for $k=1$ shows.

    Let $v\in\mathring{\mathcal V}_1^{\mathrm{DG}}(\mathcal T)$ and let $p_\tau$ be its polynomial on $\tau$. For a real number $r$, set $r^+=\max\{r,0\}$ and $r^-=\max\{-r,0\}$. Define affine functions $p_\tau^\pm$ as
    \begin{align}
        p_\tau^\pm(\bm x)
        :=\sum_{i=0}^d \bigl(p_\tau(\bm x_i)\bigr)^\pm
        \lambda_i^\tau(\bm x).
    \end{align}
    Here $p_\tau^\pm$ interpolate the positive and negative parts of the vertex values, rather than denote the pointwise positive and negative parts of $p_\tau$. Then $p_\tau=p_\tau^+-p_\tau^-$ and $p_\tau^\pm\geqslant0$ on $\overline\tau$.
    
    Fix $\varepsilon>0$ and set $q^\pm=p_\tau^\pm+\varepsilon$. For either $q=q^+$ or $q=q^-$, write $q_i=q(\bm x_i)>0$, choose $C\geqslant1+\max_iq_i$, and define
    \begin{align}
        F_q^\tau(\bm x)
        =\sum_{i=0}^d(q_i-C)\varrho_1(\lambda_i^\tau(\bm x))
        +(C-1)\sum_{i=0}^d\varrho_0(\lambda_i^\tau(\bm x))
        +1-d(C-1).
    \end{align}
    On $\tau$, substitution gives $F_q^\tau=q$. Outside $\tau$, the ReLU term is at most $-1$ and at most $d$ ReLU$^0$ functions equal one, so $F_q^\tau\leqslant0$. Hence
    \begin{align}
        \varrho_1(F_q^\tau)&=q\,\mathbf1_\tau,\\
        p_\tau\mathbf1_\tau
        &=\varrho_1(F_{q^+}^\tau)-\varrho_1(F_{q^-}^\tau).
    \end{align}
    Summing over the elements gives second-layer width $2N_\mathcal T$, with the same $2H_\mathcal T$ shared first-layer preactivations. The reduction from $(d+1)N_\mathcal T$ is strict when $d\geqslant2$, but the second-layer weights now depend on the target function. Thus Theorem~\ref{thm: fnn interior} provides an explicit width bound with reusable hidden features, rather than a lower bound for arbitrary representations.
\end{remark}

\subsection{Vector- and tensor-valued finite elements}\label{sec:tensor finite elements}

The hidden features in Theorem~\ref{thm: fnn interior} are independent of the represented function. They can therefore be shared by all components of a vector- or tensor-valued function.

\begin{corollary}[Multicomponent representation]\label{cor: dg multi}
	Let $M=\binom{d+k}{d}$. Then ${\fnn}_d^m(k;2H_\mathcal T,MN_\mathcal T)$ realizes $\mathring{\mathcal V}_k^{\mathrm{DG}}(\mathcal T;\mathbb E)$ pointwise on $\overline\Omega$. For fixed $d$ and $k$, all hidden parameters depend only on the mesh, and only the output coefficients vary with the represented function.
\end{corollary}

\begin{proof}
    {Write $\bm v=(v_1,\dots,v_m)$} with $v_i\in\mathring{\mathcal V}_k^{\mathrm{DG}}(\mathcal T)$. By Theorem~\ref{thm: fnn interior}, the scalar components share the hidden parameters $\bm W_\I$, $\bm W_{\II,0}$, $\bm W_{\II,1}$, $\bm b_\I$, and $\bm b_\II$. Only their output rows $\bm W_\III^i$ differ. Stack these rows as
    \begin{align}
        \bm W_{\III} = \begin{pmatrix}
            \bm W_\III^1\\
            \bm W_\III^2\\
            \vdots\\
            \bm W_\III^m
        \end{pmatrix}.
    \end{align}
    Then 
    \begin{align}
        \bm v(\bm x)
		=
		\bm W_\III\varrho_k\!\left(
		\bm W_{\II,0}\varrho_0(\bm W_\I\bm x+\bm b_\I)
		+
		\bm W_{\II,1}\varrho_1(\bm W_\I\bm x+\bm b_\I)
		+
		\bm b_\II
		\right).
    \end{align}
\end{proof}

Corollary~\ref{cor: dg multi} gives an almost-everywhere exact representation of every function in $\mathcal V_k^{\mathrm{DG}}(\mathcal T;\mathbb E)$. More generally, it applies to any finite element subspace
\begin{align}
	X_h\subset\mathcal V_k^{\mathrm{DG}}(\mathcal T;\mathbb E),
\end{align}
including spaces with additional local constraints or interelement trace conditions, provided $k$ bounds the actual polynomial degree of their local shape functions. For $H(\operatorname{div})$- or $H(\operatorname{curl})$-conforming spaces, choosing the zero-skeleton representative preserves the Sobolev class and its traces. Its assigned skeleton values, however, need not equal those traces. For vector fields one takes $\mathbb E=\R^d$; for symmetric tensor fields one takes the space of symmetric $d\times d$ matrices, with $m=d(d+1)/2$. Neither the output dimension nor such constraints require duplicating the hidden features.

The following examples make the required polynomial degree and output dimension explicit. All elements are affine simplices. We use $r$ for the finite element index and reserve $k$ for the degree in Corollary~\ref{cor: dg multi}; $\widetilde{\P}_r$ denotes the homogeneous polynomials of degree $r$. No essential boundary conditions are imposed in the spaces below.
\begin{itemize}
\item \emph{Raviart--Thomas elements.}
For $r\geqslant0$, the local and global spaces are
\begin{align}
\mathrm{RT}_r(\tau)
&=\P_r(\tau;\R^d)\oplus\bm x\,\widetilde{\P}_r(\tau),\\
\mathrm{RT}_r(\mathcal T)
&=\left\{\bm v\in H(\operatorname{div};\Omega):
\bm v|_\tau\in\mathrm{RT}_r(\tau)\text{ for every }\tau\in\mathcal T\right\}.
\end{align}
The global space has continuous normal traces across interior facets; see \cite{raviart1977mixed,nedelec1980mixed}. Since $\mathrm{RT}_r(\tau)\subset\P_{r+1}(\tau;\R^d)$,
\begin{align}
\mathrm{RT}_r(\mathcal T)\subset
\mathcal V_{r+1}^{\mathrm{DG}}(\mathcal T;\R^d).
\end{align}
Thus Corollary~\ref{cor: dg multi} applies with $(k, m)=(r+1, d)$.

\item \emph{N\'{e}d\'{e}lec elements.}
In dimensions $d=2,3$, the first- and second-family local spaces can be written as
\begin{align}
\mathrm{N}_r^{\mathrm I}(\tau)
&=\P_r(\tau;\R^d)\oplus
\left\{\bm p\in\widetilde{\P}_{r+1}(\tau;\R^d):
\bm p(\bm x)\cdot\bm x=0\right\},\quad r\geqslant0,\\
\mathrm{N}_r^{\mathrm{II}}(\tau)
&=\P_r(\tau;\R^d),\quad r\geqslant1.
\end{align}
For either family $a\in\{\mathrm I,\mathrm{II}\}$, the global space is
\begin{align}
\mathrm{N}_r^a(\mathcal T)
=\left\{\bm v\in H(\operatorname{curl};\Omega):
\bm v|_\tau\in\mathrm{N}_r^a(\tau)\text{ for every }\tau\in\mathcal T\right\}.
\end{align}
These spaces have continuous tangential traces \cite{nedelec1980mixed,nedelec1986new}. The two families are contained in $\mathcal V_{r+1}^{\mathrm{DG}}(\mathcal T;\R^d)$ and $\mathcal V_r^{\mathrm{DG}}(\mathcal T;\R^d)$, respectively. Hence the required pairs $(k,m)$ are $(r+1,d)$ and $(r,d)$. The first-family index used here starts at $r=0$; conventions that start at order one shift this index by one. Higher-dimensional counterparts are described by polynomial differential forms, with the same componentwise representation principle \cite{arnold2010feec}.

\item \emph{Hu--Zhang symmetric stress elements.}
Let $\mathbb S=\{A\in\R^{d\times d}:A^{\mathrm T}=A\}$. For $r\geqslant d+1$, the Hu--Zhang construction and its extension to arbitrary dimension \cite{huzhang2014triangular,hu2015higher} use the local space $\P_r(\tau;\mathbb S)$. Define its zero-traction bubble subspace by
\begin{align}
B_r(\tau)
=\left\{\bm\eta\in\P_r(\tau;\mathbb S):
\bm\eta\bm n=0\text{ on }\partial\tau\right\},
\end{align}
where $\bm n$ is the outward unit normal on each facet. The global stress space has the structure
\begin{align}
\Sigma_r^{\mathrm{HZ}}(\mathcal T)
&=\mathcal V_r(\mathcal T;\mathbb S)
+\bigl\{\bm\eta_b\in\mathcal V_r^{\mathrm{DG}}(\mathcal T;\mathbb S): \bm\eta_b|_\tau\in B_r(\tau)\text{ for every }\tau\in\mathcal T\bigr\}\notag\\
&\subset H(\operatorname{div};\Omega;\mathbb S)
\cap\mathcal V_r^{\mathrm{DG}}(\mathcal T;\mathbb S).
\end{align}
The continuous part and the zero-traction bubbles ensure that $\bm\eta\bm n$ is continuous across interior facets. Corollary~\ref{cor: dg multi} therefore applies with $(k, m)=(r, d(d+1)/2)$. The accompanying displacement space is $\mathcal V_{r-1}^{\mathrm{DG}}(\mathcal T;\R^d)$, represented with $(k,m)=(r-1,d)$. For the lower-order construction with $1\leqslant r\leqslant d$, the additional facet bubbles in \cite{huzhang2016lower} have degree at most $d+1$; the enriched stress space is consequently contained in $\mathcal V_{d+1}^{\mathrm{DG}}(\mathcal T;\mathbb S)$ and is covered by taking $k=d+1$.
\end{itemize}

The exact representation of discontinuous finite element functions transfers elementwise polynomial approximation estimates to the network class. We equip $\mathbb E$ with the Euclidean norm in the chosen basis, denoted by $\norm{\cdot}_{\mathbb E}$, and use the corresponding norms and seminorms on $L^p(\Omega;\mathbb E)$ and $W^{s,p}(\Omega;\mathbb E)$. For an integer $s\geqslant0$, define the broken Sobolev norm by
\begin{align}\label{eq:broken Sobolev norm}
	\norm{\bm w}_{W^{s,p}(\mathcal T;\mathbb E)}
	:=
	\begin{cases}
		\displaystyle
		\left(
		\sum_{\tau\in\mathcal T}
		\norm{\bm w}_{W^{s,p}(\tau;\mathbb E)}^p
		\right)^{1/p},
		& 1\leqslant p<\infty,\\[2ex]
		\displaystyle
		\max_{\tau\in\mathcal T}
		\norm{\bm w}_{W^{s,\infty}(\tau;\mathbb E)},
		& p=\infty.
	\end{cases}
\end{align}
All derivatives in this norm are taken within individual elements.

\begin{corollary}[Broken Sobolev approximation]
	\label{cor:DG Sobolev approximation}
	Let $k\geqslant0$, $1\leqslant p\leqslant\infty$, and let $1\leqslant r\leqslant k+1$ be an integer. Let $M=\binom{d+k}{d}$ and $\{\mathcal T_h\}$ be a shape-regular family of conforming simplicial meshes of a bounded polyhedral domain $\Omega$, with $h=\max_{\tau\in\mathcal T_h}\operatorname{diam}(\tau) \leqslant1$. For every $\bm v\in W^{r,p}(\Omega;\mathbb E)$, there exists a network $\bm{\mathcal N}\in \fnn_d^m\left(k;2H_{\mathcal T_h},MN_{\mathcal T_h}\right)$ such that, for any $s=0,\ldots,r$,
	\begin{align}\label{eq:DG Sobolev approximation}
		\norm{\bm v-\bm{\mathcal N}}_{W^{s,p}(\mathcal T_h;\mathbb E)}
		\leqslant
		C h^{r-s}\abs{\bm v}_{W^{r,p}(\Omega;\mathbb E)}.
	\end{align}
	The constant $C$ depends only on $d$, $k$, $r$, $p$, $m$, and the shape-regularity bound, and is independent of $h$ and $\bm v$.
\end{corollary}

\begin{proof}
	For each $\tau\in\mathcal T_h$, apply the averaged Taylor polynomial construction and the Bramble--Hilbert estimate in \cite[Lemma~4.3.8]{brenner2008mathematical} componentwise. Writing $h_\tau=\operatorname{diam}(\tau)$, we obtain $\bm q_\tau\in\mathbb P_{r-1}(\tau;\mathbb E)$ such that
	\begin{align}\label{eq:local DG polynomial approximation}
		\abs{\bm v-\bm q_\tau}_{W^{j,p}(\tau;\mathbb E)}
		\leqslant
		C h_\tau^{r-j}
		\abs{\bm v}_{W^{r,p}(\tau;\mathbb E)},
		\qquad j=0,\ldots,r,
	\end{align}
	where the order-zero seminorm denotes the $L^p$ norm. These estimates hold for all $1\leqslant p\leqslant\infty$. Shape regularity makes the constant uniform over the mesh family.
	
	Since $r-1\leqslant k$, define $\bm v_h\in\mathring{\mathcal V}_k^{\mathrm{DG}} (\mathcal T_h;\mathbb E)$ by setting $\bm v_h|_\tau=\bm q_\tau$ on each open element and $\bm v_h=0$ on the mesh skeleton. Corollary~\ref{cor: dg multi} provides a network that agrees pointwise with $\bm v_h$ on $\overline\Omega$. In particular, its restriction to each open element equals $\bm q_\tau$, together with all elementwise weak	derivatives.
	
	For $0\leqslant j\leqslant s\leqslant r$, the inequalities $h_\tau\leqslant h\leqslant1$ give $h_\tau^{r-j}\leqslant h^{r-s}$. Combining \eqref{eq:local DG polynomial approximation} for $j=0,\ldots,s$ and summing over the elements yields \eqref{eq:DG Sobolev approximation} when $p<\infty$. For $p=\infty$, the same argument uses the maximum over the elements. The assigned skeleton values do not affect these broken Sobolev norms.
\end{proof}

\subsection{Pointwise representation with prescribed skeleton values}\label{sec:hat DG representation}

To represent prescribed polynomial values on the skeleton as well, we now consider $\widehat{\mathcal V}_k^{\mathrm{DG}}(\mathcal T;\mathbb E)$. Its polynomial data on distinct relatively open subsimplices may be chosen independently. We localize a polynomial to each such subsimplex and sum the resulting contributions. The essential building block is the following localization result.

\begin{lemma}[Localization to subsimplices]\label{lem:single subsimplex localizer}
	Let $\sigma$ be a relatively open subsimplex of $\overline\tau$, where $\tau$ is a $d$-simplex, and let $\bm\beta\in\mathcal A_k^d$. There exists a second-layer preactivation formed from the two activation channels of the $2(d+1)$ first-layer affine preactivations $\bm\lambda^\tau(\bm x)$ and $-\bm\lambda^\tau(\bm x)$,
	\begin{align}
		f_{\bm\beta}^{\tau,\sigma}(\bm x)
		={}&
		(k+1)\sum_{i\in\mathcal I_{\tau,\sigma}}
		\varrho_0(\lambda_i^\tau(\bm x))
		+(k+1)\sum_{i\notin\mathcal I_{\tau,\sigma}}
		\left(
		1-
		\varrho_0(\lambda_i^\tau(\bm x))-
		\varrho_0(-\lambda_i^\tau(\bm x))
		\right)\notag\\
		&+
		(\bm\beta^\circ-(k+1)\bm1)
		\cdot\varrho_1(\bm\lambda^\tau(\bm x))
		+1-(k+1)d,
	\end{align}
	such that
	\begin{align}
		f_{\bm\beta}^{\tau,\sigma}(\bm x)
		=L_{\bm\beta}^\tau(\bm x)
		\quad\text{for }\bm x\in\sigma,
		\qquad
		f_{\bm\beta}^{\tau,\sigma}(\bm x)
		\leqslant0
		\quad\text{for }\bm x\notin\sigma.
	\end{align}
Consequently,
	\begin{align}\label{eq:single subsimplex basis localizer}
		\varrho_k(f_{\bm\beta}^{\tau,\sigma}(\bm x))
		=
		\begin{cases}
		(L_{\bm\beta}^\tau(\bm x))^k,&\bm x\in\sigma,\\
		0,&\bm x\notin\sigma.
		\end{cases}
	\end{align}
\end{lemma}

\begin{proof}
    Use the tests
    \begin{align}
        {\chi_i^\tau}(\bm x)=
        \begin{cases}
            \varrho_0(\lambda_i^\tau(\bm x)),&i\in\mathcal I_{\tau,\sigma},\\
            1-\varrho_0(\lambda_i^\tau(\bm x))
              -\varrho_0(-\lambda_i^\tau(\bm x)),&i\notin\mathcal I_{\tau,\sigma}.
        \end{cases}
    \end{align}
    The first case tests $\lambda_i^\tau>0$ and the second tests $\lambda_i^\tau=0$. By \eqref{eq:subsimplex bcf}, all $d+1$ tests equal one exactly on $\sigma$. The common ReLU term is
    \begin{align}\label{eq:subsimplex g bound}
        g_{\bm\beta}^{\tau,\sigma}
        =(\bm\beta^\circ-(k+1)\bm1)\cdot\varrho_1(\bm\lambda^\tau)
        \leqslant-1
        \quad\text{on }\R^d.
    \end{align}
    Thus \eqref{eq:indicator localizer} gives the stated values of $f_{\bm\beta}^{\tau,\sigma}$. Since $L_{\bm\beta}^\tau\geqslant1$ on $\overline\tau$ and $\varrho_k(t)=0$ for $t\leqslant0$, \eqref{eq:single subsimplex basis localizer} follows, including when $k=0$.
\end{proof}

We now assign every relatively open subsimplex $\sigma$ to an element $\tau$ and apply Lemma~\ref{lem:single subsimplex localizer} separately to each subsimplex.

\begin{theorem}[Prescribed-skeleton representation]\label{thm:subsimplex representation}
	{Set $M=\binom{d+k}{d}$.} Then ${\fnn}_d^m(k;2H_\mathcal T,MN_\mathcal S)$ realizes $\widehat{\mathcal V}_k^{\mathrm{DG}}(\mathcal T;\mathbb E)$ pointwise on $\overline\Omega$. For fixed $d$ and $k$, the hidden parameters depend only on the mesh and a fixed choice of an incident element for each subsimplex; only the output coefficients vary with the represented function. Moreover, $MN_\mathcal S\leqslant M(2^{d+1}-1)N_\mathcal T$.
\end{theorem}

\begin{proof}
	For each $\sigma\in\mathcal S_\mathcal T$, fix an element $\tau_\sigma\in\mathcal T$ such that $\sigma\subset\overline{\tau_\sigma}$. If $\sigma$ is itself a $d$-simplex, take $\tau_\sigma=\sigma$. {Let $\bm v\in\widehat{\mathcal V}_k^{\mathrm{DG}}(\mathcal T;\mathbb E)$.} For each $\sigma$, fix a linear extension from $\P_k(\R^d)|_\sigma$ to $\P_k(\R^d)$ and apply it componentwise to $\bm v|_\sigma$. Such an extension exists: choose a basis of the restriction space and fix a polynomial extension of each basis function. Denote the resulting polynomial by $\bm p_\sigma\in\P_k(\R^d;\R^m)$. Then $\bm p_\sigma|_\sigma=\bm v|_\sigma$, and its coefficients depend linearly on the data on $\sigma$. No agreement with the data on other subsimplices is required. Applying Lemma~\ref{lem:k-linear decompose of Pk} componentwise on $\tau_\sigma$ supplies coefficients $\bm u_{\bm\beta}^\sigma\in\R^m$ such that
	\begin{align}
		\bm p_\sigma(\bm x)
		=
		\sum_{\bm\beta\in\mathcal A_k^d}
		\bm u_{\bm\beta}^\sigma
		(L_{\bm\beta}^{\tau_\sigma}(\bm x))^k.
	\end{align}
	For every $\bm x\in\overline\Omega$, define
	\begin{align}\label{eq:subsimplex representation}
		\bm{\mathcal N}(\bm x)
		=
		\sum_{\sigma\in\mathcal S_\mathcal T}
		\sum_{\bm\beta\in\mathcal A_k^d}
		\bm u_{\bm\beta}^\sigma
		\varrho_k\!\left(
		f_{\bm\beta}^{\tau_\sigma,\sigma}(\bm x)
		\right).
	\end{align}
	The relatively open subsimplices form a disjoint partition of $\overline\Omega$. Hence $\bm x$ belongs to a unique $\sigma_{\bm x}\in\mathcal S_\mathcal T$. By Lemma~\ref{lem:single subsimplex localizer}, all terms in \eqref{eq:subsimplex representation} associated with $\sigma\neq\sigma_{\bm x}$ vanish, whereas the terms associated with $\sigma_{\bm x}$ give $\bm{\mathcal N}(\bm x)=\bm p_{\sigma_{\bm x}}(\bm x)=\bm v(\bm x)$.
	Thus the representation is pointwise exact on $\overline\Omega$, including the entire interior and exterior mesh skeleton.

	It remains to count the layer widths. For every unoriented mesh hyperplane, include its two affine orientations in the first-layer preactivation vector. After positive rescaling and merging duplicates, this requires at most $2H_\mathcal T$ affine preactivations. Every function $f_{\bm\beta}^{\tau_\sigma,\sigma}$ is an affine combination of the $\varrho_0$ and $\varrho_1$ channels of these preactivations, together with a bias. Therefore one second-layer neuron computes each quantity $\varrho_k(f_{\bm\beta}^{\tau_\sigma,\sigma})$, and there are $M$ choices of $\bm\beta$ for each of the $N_\mathcal S$ subsimplices. The second hidden layer consequently has width $MN_\mathcal S$, independently of $m$. The coefficient vectors $\bm u_{\bm\beta}^\sigma$ form the columns of the output matrix $\bm W_\III\in\R^{m\times MN_\mathcal S}$. Thus all components share the same hidden features, and the output layer forms the required vector-valued linear combination. The bound on $N_\mathcal S$ follows from Section~\ref{sec:simplicial meshes}.
\end{proof}

For compatible finite elements, this theorem can also realize a chosen representative whose skeleton values are polynomial on each subsimplex. Such pointwise data must be specified consistently as a single value on each subsimplex. The theorem does not identify arbitrary skeleton assignments with the one-sided polynomial traces of incident elements; any required agreement with normal, tangential, or traction traces must be imposed on the chosen data.

\section{Pointwise realization of continuous finite element functions}\label{sec:general continuous construction}

\subsection{Continuous representations}\label{sec:continuous representations}

The continuous representation and approximation results below apply to scalar-, vector-, and tensor-valued functions. All components share the same hidden features, so the hidden-layer widths are independent of $m$; only the output dimension changes. The storage analysis in Section~\ref{sec:storage} is restricted to the scalar case $m=1$.

The pointwise representation of continuous finite element functions on a general mesh follows directly from Theorem~\ref{thm:subsimplex representation}, because their restrictions to all relatively open subsimplices are polynomials of the prescribed degree.

\begin{corollary}[Continuous representation]\label{cor:continuous subsimplex representation}
	Let $M=\binom{d+k}{d}$. Then ${\fnn}_d^m(k;2H_\mathcal T,MN_\mathcal S)$ realizes $\mathcal V_k(\mathcal T;\mathbb E)$ pointwise on $\overline\Omega$, including all skeleton values. Moreover, $MN_\mathcal S\leqslant M(2^{d+1}-1)N_\mathcal T$.
\end{corollary}

\begin{proof}
	Apply Theorem~\ref{thm:subsimplex representation} and use the inclusion $\mathcal V_k(\mathcal T;\mathbb E)\subset\widehat{\mathcal V}_k^{\mathrm{DG}}(\mathcal T;\mathbb E)$ from Section~\ref{sec:finite element spaces}.
\end{proof}

\paragraph{A smaller construction via half-open decompositions.}
\begin{definition}[Half-open decomposition]\label{def:half open decomposition}
For each $\tau\in\mathcal T$, choose complementary index sets $\mathcal I_\tau^+$ and $\mathcal I_\tau^-$, either of which may be empty, such that
$\mathcal I_\tau^+\sqcup\mathcal I_\tau^-=\{0,\ldots,d\}$, and set
\begin{align}\label{eq:assigned element characterization}
    \widetilde\tau
    =
    \Big\{\bm x\in\R^d:
    \ \lambda_i^\tau(\bm x)\geqslant0
    \text{ for }i\in\mathcal I_\tau^+;
    \ \lambda_i^\tau(\bm x)>0
    \text{ for }i\in\mathcal I_\tau^-
    \Big\}.
\end{align}
Thus $\widetilde\tau$ is obtained from $\overline\tau$ by removing the closed facets indexed by $\mathcal I_\tau^-$. The family $\{\widetilde\tau:\tau\in\mathcal T\}$ is called a half-open decomposition of $\mathcal T$ if
\begin{align}\label{eq:half open decomposition}
    \overline\Omega=\bigsqcup_{\tau\in\mathcal T}\widetilde\tau.
\end{align}
\end{definition}

Since incident element polynomials agree on shared faces, a half-open
decomposition lets each element polynomial supply the values on its
assigned boundary faces as well as its interior. Thus no separate
expansions on lower-dimensional subsimplices are needed, reducing the
second-layer width from $MN_\mathcal S$ to $MN_\mathcal T$.
We first construct this representation using \eqref{eq:indicator localizer},
then give a sufficient geometric condition for the decomposition to exist.

\begin{lemma}[Localization to half-open simplices]\label{lem:assigned element localizer}
	For each $\tau\in\mathcal T$ and $\bm\beta\in\mathcal A_k^d$, define
	\begin{align}
		\chi_i^{\tau}(\bm x)
		=
		\begin{cases}
			1-\varrho_0(-\lambda_i^\tau(\bm x)),
			&i\in\mathcal I_{\tau}^+,\\
			\varrho_0(\lambda_i^\tau(\bm x)),
			&i\in\mathcal I_{\tau}^-
		\end{cases}{.}
	\end{align}
	The corresponding second-layer preactivation is
	\begin{align}\label{eq:assigned element localizer}
		f_{\bm\beta}^{\tau}(\bm x)
		=
		(\bm\beta^\circ-(k+1)\bm1)
		\cdot\varrho_1(\bm\lambda^\tau(\bm x))
		+(k+1)\sum_{i=0}^d
		\chi_i^{\tau}(\bm x)
		+1-(k+1)d.
	\end{align}
	It satisfies
	\begin{align}\label{eq:assigned element localizer property}
		f_{\bm\beta}^{\tau}(\bm x)
		=
		L_{\bm\beta}^\tau(\bm x)
		\quad\text{for }\bm x\in\widetilde\tau,
		\qquad
		f_{\bm\beta}^{\tau}(\bm x)
		\leqslant0
		\quad\text{for }\bm x\notin\widetilde\tau.
	\end{align}
\end{lemma}

\begin{proof}
	The identities $1-\varrho_0(-t)=\mathbf1_{\{t\geqslant0\}}$ and $\varrho_0(t)=\mathbf1_{\{t>0\}}$ show that the $\chi_i^{\tau}$ test the conditions in \eqref{eq:assigned element characterization}. All $d+1$ tests equal one exactly on $\widetilde\tau\subset\overline\tau$. Formula~\eqref{eq:assigned element localizer} is therefore a case of \eqref{eq:indicator localizer}, which gives both claims.
\end{proof}

\begin{theorem}[Continuous representation via half-open decompositions]\label{thm:continuous FE representation}
	Let $M=\binom{d+k}{k}$, and suppose that $\mathcal T$ admits a half-open decomposition as in Definition~\ref{def:half open decomposition}. Then $\fnn_d^m(k;2H_\mathcal T,MN_\mathcal T)$ realizes $\mathcal V_k(\mathcal T;\mathbb E)$ pointwise on $\overline\Omega$. For fixed $d$ and $k$, the hidden parameters depend only on the mesh and the chosen half-open decomposition. All components share these parameters, and only the output coefficients vary with the represented function.
\end{theorem}

The statement includes $k=0$: then $M=1$, and each assigned region requires one second-layer $\varrho_0$ neuron.

\begin{proof}
	{Let $\bm v\in\mathcal V_k(\mathcal T;\mathbb E)$.} Applying Lemma~\ref{lem:k-linear decompose of Pk} componentwise, for each $\tau\in\mathcal T$ write
	\begin{align}
		\bm v|_\tau
		=
		\bm p_\tau
		=
		\sum_{\bm\beta\in\mathcal A_k^d}
		\bm u_{\bm\beta}^\tau
		(L_{\bm\beta}^\tau)^k,
	\end{align}
	where $\bm p_\tau\in\P_k(\R^d;\R^m)$ and $\bm u_{\bm\beta}^\tau\in\R^m$. The functions $\bm v|_{\overline\tau}$ and $\bm p_\tau|_{\overline\tau}$ are continuous and agree on the dense subset $\tau$ of $\overline\tau$. Hence $\bm p_\tau=\bm v$ on the entire closure $\overline\tau$. Define
	\begin{align}\label{eq:continuous assigned element representation}
		\bm{\mathcal N}(\bm x)
		=
		\sum_{\tau\in\mathcal T}
		\sum_{\bm\beta\in\mathcal A_k^d}
		\bm u_{\bm\beta}^\tau
		\varrho_k\!\left(
		f_{\bm\beta}^{\tau}(\bm x)
		\right).
	\end{align}
	For every $\bm x\in\overline\Omega$, the disjoint partition \eqref{eq:half open decomposition} gives a unique element $\tau_{\bm x}$ such that $\bm x\in\widetilde{\tau_{\bm x}}$. Lemma~\ref{lem:assigned element localizer} then shows that all contributions in \eqref{eq:continuous assigned element representation} except those associated with $\tau_{\bm x}$ vanish. Since $\bm x\in\overline{\tau_{\bm x}}$, the remaining terms give $\bm{\mathcal N}(\bm x)=\bm p_{\tau_{\bm x}}(\bm x)=\bm v(\bm x)$.
	Thus the representation is pointwise exact on $\overline\Omega$, including the exterior boundary and the interior mesh skeleton.

	The first hidden layer contains both orientations of every mesh hyperplane and therefore has at most $2H_\mathcal T$ affine preactivations. For every pair $(\tau,\bm\beta)$, one second-layer neuron computes $\varrho_k(f_{\bm\beta}^{\tau})$. There are $M$ such neurons per element and hence $MN_\mathcal T$ in total. The coefficient vectors $\bm u_{\bm\beta}^\tau$ form the columns of $\bm W_\III\in\R^{m\times MN_\mathcal T}$. Thus the hidden layers are shared by all components, and their widths do not depend on $m$. The construction therefore has exactly the two hidden layers specified in Definition~\ref{def:FNN}.
\end{proof}

To apply the preceding theorem, it remains to ensure the existence of a half-open decomposition. A sufficient condition is the following star condition: there exists a point $\bm z\in\Omega$, lying on no mesh hyperplane, such that
\begin{align}\label{eq:star shaped assumption}
    (1-t)\bm x+t\bm z\in\Omega
    \qquad\text{for every }\bm x\in\overline\Omega\text{ and }0<t\leqslant1.
\end{align}
This condition holds on every bounded convex polyhedral domain. The following lemma adapts the classical half-open decomposition of \cite[Corollary~5.3.5]{beck2018combinatorial} to domains satisfying \eqref{eq:star shaped assumption}.

\begin{lemma}[Half-open decomposition under the star condition]\label{lem:subsimplex element assignment}
Assume that $\bm z\in\Omega$ lies on no mesh hyperplane and satisfies \eqref{eq:star shaped assumption}. For each $\tau\in\mathcal T$, set
\begin{align}
    \mathcal I_{\tau,\bm z}^+=\{i:\lambda_i^\tau(\bm z)>0\},
    \qquad
    \mathcal I_{\tau,\bm z}^-=\{i:\lambda_i^\tau(\bm z)<0\},
\end{align}
and let $\widetilde\tau^{\bm z}$ be the set defined by \eqref{eq:assigned element characterization} with these index sets. Then $\{\widetilde\tau^{\bm z}:\tau\in\mathcal T\}$ is a half-open decomposition of $\mathcal T$.

More precisely, for every $\sigma\in\mathcal S_\mathcal T$ and any $\bm x_\sigma\in\sigma$, the point $(1-\varepsilon)\bm x_\sigma+\varepsilon\bm z$ lies in a fixed open element $a_{\bm z}(\sigma)$ for all sufficiently small $\varepsilon>0$. This element is independent of $\bm x_\sigma$, satisfies $\sigma\subset\overline{a_{\bm z}(\sigma)}$, and equals $\sigma$ when $\sigma\in\mathcal T$. Moreover,
\begin{align}\label{eq:assigned element region}
    \widetilde\tau^{\bm z}
    =
    \bigsqcup_{\substack{\sigma\in\mathcal S_\mathcal T\\a_{\bm z}(\sigma)=\tau}}\sigma.
\end{align}
\end{lemma}

\begin{proof}
For $\bm x\in\overline\Omega$, the star condition and the genericity of $\bm z$ ensure that $(1-\varepsilon)\bm x+\varepsilon\bm z$ remains in $\Omega$ and avoids every mesh hyperplane for all sufficiently small $\varepsilon>0$. It therefore lies in a fixed open element $\tau$, with $\bm x\in\overline\tau$. For $\bm x\in\overline\tau$, the identity
\begin{align}
    \lambda_i^\tau((1-\varepsilon)\bm x+\varepsilon\bm z)
    =(1-\varepsilon)\lambda_i^\tau(\bm x)
    +\varepsilon\lambda_i^\tau(\bm z)
\end{align}
shows that the perturbed point lies in $\tau$ for all sufficiently small $\varepsilon>0$ if and only if $\lambda_i^\tau(\bm z)>0$ whenever $\lambda_i^\tau(\bm x)=0$. By conformity, the relatively open subsimplex $\sigma$ containing $\bm x$ is contained in $\overline\tau$, and each $\lambda_i^\tau$ is either positive throughout $\sigma$ or identically zero there. Thus the criterion is independent of the representative in $\sigma$, proving the asserted properties of $a_{\bm z}$. The same criterion is precisely \eqref{eq:assigned element characterization} with $\mathcal I_\tau^\pm=\mathcal I_{\tau,\bm z}^\pm$, which gives \eqref{eq:assigned element region}. Since every point is assigned to exactly one element, these regions satisfy \eqref{eq:half open decomposition}.
\end{proof}

Applying Theorem~\ref{thm:continuous FE representation} to the half-open
decomposition supplied by Lemma~\ref{lem:subsimplex element assignment}
yields the following corollary.

\begin{corollary}[Continuous representation under the star condition]\label{cor:continuous star representation}
Let $M=\binom{d+k}{d}$. If $\Omega$ satisfies \eqref{eq:star shaped assumption} for a point $\bm z$ lying on no mesh hyperplane, then $\fnn_d^m(k;2H_\mathcal T,MN_\mathcal T)$ realizes $\mathcal V_k(\mathcal T;\mathbb E)$ pointwise on $\overline\Omega$. The hidden parameters depend only on the mesh, the degree, and the choice of $\bm z$.
\end{corollary}

\begin{figure}[htbp]
    \centering
    \includegraphics[width=.35\textwidth]{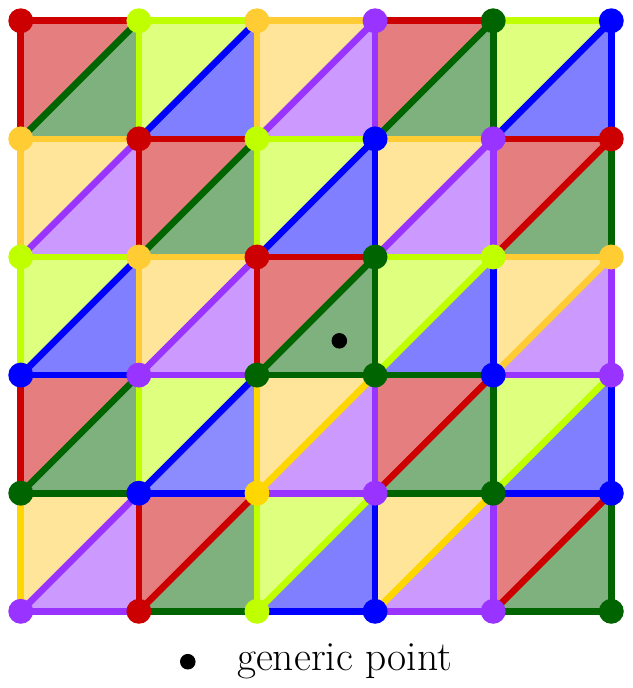}
    \caption{A half-open decomposition on a Kuhn triangulation.}
    \label{fig:disjoint partition}
\end{figure}

Figure~\ref{fig:disjoint partition} illustrates the assignment toward the generic point $\bm z$, marked by the black dot. The bold subsimplices are assigned to the open triangles indicated by the corresponding shading. Each $\widetilde\tau^{\bm z}$ thus consists of $\tau$ together with its assigned boundary subsimplices, as in \eqref{eq:assigned element region}.

\begin{remark}[Necessity of additional assumptions]\label{rem:half open obstruction}
A half-open decomposition need not exist on an arbitrary mesh. To see this in dimension two, let $n_j=n_j(\widetilde\tau)$ denote the number of relatively open $j$-dimensional faces contained in $\widetilde\tau$. Removing $r=\abs{\mathcal I_\tau^-}$ closed edges from $\overline\tau$ gives
\begin{align}
    \begin{array}{c|cccc}
        r & 0 & 1 & 2 & 3 \\
        \hline
        (n_0,n_1,n_2) & (3,3,1) & (1,2,1) & (0,1,1) & (0,0,1) \\
        n_0-n_1+n_2 & 1 & 0 & 0 & 1
    \end{array}
\end{align}
In a half-open decomposition, each relatively open mesh face belongs to exactly one $\widetilde\tau$. Hence, writing $V$, $E$, and $F$ for the total numbers of mesh vertices, edges, and triangles, respectively, we obtain
\begin{align}
    \chi(\overline\Omega)
    &=V-E+F
    =\sum_{\tau\in\mathcal T}
    \bigl(n_0(\widetilde\tau)-n_1(\widetilde\tau)+n_2(\widetilde\tau)\bigr)\notag\\
    &=\#\bigl\{\tau\in\mathcal T:\abs{\mathcal I_\tau^-}\in\{0,3\}\bigr\}
    \geqslant0.
\end{align}
However, a connected planar polygonal domain with two holes has $\chi(\overline\Omega)=1-2=-1$, a contradiction. Thus no triangulation of such a domain admits a half-open decomposition of the stated form, and additional assumptions are necessary to guarantee the decomposition used for the width reduction.
\end{remark}

\begin{remark}[Beyond the star condition]
The star condition is sufficient but not necessary. A simplicial isomorphism preserves half-open decompositions: its affine restriction to each simplex preserves barycentric coordinates. Thus the decomposition depends on the mesh's combinatorial structure rather than on a particular geometric realization. In fact, in dimension two, every conforming triangulation of a polygonal domain whose closure is homeomorphic to a closed disk or a closed annulus admits a half-open decomposition of the form in Definition~\ref{def:half open decomposition}; this follows from the partitionability results in \cite[Section~4.1 and Proposition~4.4(1)]{santamaria2022partitioning}. In particular, this covers domains bounded by a simple closed polygonal curve, as well as domains between two disjoint nested simple closed polygonal curves. The domain may be nonconvex and need not satisfy the star condition. Thus Theorem~\ref{thm:continuous FE representation} gives pointwise representation on $\overline\Omega$ with second-layer width $MN_\mathcal T$, using the original mesh.
\end{remark}

\begin{remark}[Interior representation]\label{rem:continuous interior representation}
If values on $\partial\Omega$ are not prescribed, any conforming simplicial mesh admits a pointwise representation of $\mathcal V_k(\mathcal T;\mathbb E)$ on $\Omega$ by networks in
\begin{align}
    \fnn_d^m(k;2H_\mathcal T,MN_\mathcal T),
    \qquad M=\binom{d+k}{d}.
\end{align}
This includes all interior skeleton values and requires no star condition on $\Omega$.

To see this, enclose $\overline\Omega$ in a bounded convex polyhedral domain $D$ and extend $\mathcal T$ to a conforming simplicial mesh $\mathcal T_D$ of $D$, keeping the original elements unchanged and allowing additional vertices outside $\overline\Omega$. Choose $\bm z\in D$ on no hyperplane of $\mathcal T_D$. Lemma~\ref{lem:subsimplex element assignment} gives a half-open decomposition of $\overline D$. For every $\bm x\in\Omega$, the point $(1-\varepsilon)\bm x+\varepsilon\bm z$ remains in $\Omega$ for all sufficiently small $\varepsilon>0$, so its assigned element belongs to $\mathcal T$. Consequently,
\begin{align}
    \Omega=\bigsqcup_{\tau\in\mathcal T}
    (\widetilde\tau^{\bm z}\cap\Omega).
\end{align}
The localization and assembly in the proof of Theorem~\ref{thm:continuous FE representation}, restricted to the original elements, therefore reproduce the function throughout $\Omega$. Only the original mesh hyperplanes and one $M$-term expansion per original element are needed, giving the stated descriptor. Boundary points may instead be assigned to added elements, so no assertion is made about the values on $\partial\Omega$.
\end{remark}

\subsection{Width and sparse storage}\label{sec:storage}
Both continuous constructions use at most $2H_\mathcal T$ shared affine preactivations, or $4H_\mathcal T$ scalar neurons in the first hidden layer. The second-layer width is $MN_\mathcal S$ on a general mesh and $MN_\mathcal T$ under the star condition. These widths are independent of the output dimension. The parameter conversion in \eqref{eq:MLP parameter conversion} preserves these scalar-neuron widths and the linear sparse-storage scaling. In the rest of this subsection, we restrict attention to the scalar case $\mathbb E=\R$ and $m=1$, and count the stored entries in the general construction.

Fix $d,k\geqslant1$. For a conforming simplicial mesh $\mathcal T$, we count storage using the prescribed sparsity pattern of the implementation in Theorem~\ref{thm:subsimplex representation} with $m=1$. For a matrix or vector $A$, let $\nslot(A)$ denote the number of retained positions, including those whose stored values happen to be zero. Let $P_{\mathcal T}$ denote the total number of retained positions in all weights and biases.

\begin{corollary}[Sparse storage and finite element dimension]
\label{cor:dof optimality}
Under the preceding storage convention,
\begin{equation}\label{eq:element sparse scaling}
    P_{\mathcal T}\asymp_{d,k}N_{\mathcal T}.
\end{equation}
If $\mathcal T$ is shape-regular with all element solid angles at vertices bounded below
by a mesh-independent constant $\theta_0>0$, then additionally
\begin{equation}\label{eq:dof optimal sparse scaling}
    P_{\mathcal T}
    \asymp_{d,k,\theta_0}
    \dim\mathcal V_k(\mathcal T).
\end{equation}
\end{corollary}

\begin{proof}
	Set $M=\binom{d+k}{d}$. The elementary bounds
	\begin{align}
		H_{\mathcal T}\leqslant(d+1)N_{\mathcal T},
		\qquad
		N_{\mathcal T}\leqslant N_{\mathcal S}
		\leqslant(2^{d+1}-1)N_{\mathcal T}
	\end{align}
	and Theorem~\ref{thm:subsimplex representation} give
	\begin{align}
		n_1=2H_{\mathcal T}=\mathcal O_d(N_{\mathcal T}),
		\qquad
		n_2=MN_{\mathcal S}=\mathcal O_{d,k}(N_{\mathcal T}).
	\end{align}
	Each row of $\bm W_{\II,0}$ retains at most $2(d+1)$ storage slots, while each row of $\bm W_{\II,1}$ retains at most $d+1$. Therefore
	\begin{align}
		\nslot(\bm W_\I)+
		\nslot(\bm b_\I)
		&=\mathcal O_d(H_{\mathcal T})
		=\mathcal O_d(N_{\mathcal T}),\\
		\nslot(\bm W_{\II,0})+
		\nslot(\bm W_{\II,1})+
		\nslot(\bm b_\II)+
		\nslot(\bm W_\III)
		&=\mathcal O_{d,k}(MN_{\mathcal S})
		=\mathcal O_{d,k}(N_{\mathcal T}).
	\end{align}
	Thus $P_{\mathcal T}\lesssim_{d,k}N_{\mathcal T}$. Conversely, note
	\begin{align}\label{eq:sparse storage lower bound}
		P_{\mathcal T} \geqslant \nslot(\bm W_\III) = n_2=MN_{\mathcal S}
		\geqslant MN_{\mathcal T}.
	\end{align}
	It follows that $P_{\mathcal T}\asymp_{d,k}N_{\mathcal T}$.

	It remains to compare $N_{\mathcal T}$ with $\dim\mathcal V_k(\mathcal T)$ under the additional shape regularity condition. The restriction map from $\mathcal V_k(\mathcal T)$ into the product of the local polynomial spaces is injective, and therefore
	\begin{align}\label{eq:finite element dimension upper bound}
		\dim\mathcal V_k(\mathcal T)
		\leqslant MN_{\mathcal T}.
	\end{align}
	Let $V$ be the number of mesh vertices, and let $m_{\bm y}$ be the number of elements incident to a vertex $\bm y$. By the shape regularity condition, $m_{\bm y}\leqslant C_{d,\theta_0}$ for every mesh vertex $\bm y$. Double counting the element--vertex incidences yields
	\begin{align}
		(d+1)N_{\mathcal T}
		=
		\sum_{\bm y\,\text{a mesh vertex}}m_{\bm y}
		\leqslant
		C_{d,\theta_0}V.
	\end{align}
	Because $k\geqslant1$, $\mathcal V_1(\mathcal T)\subset\mathcal V_k(\mathcal T)$ and $\dim\mathcal V_1(\mathcal T)=V$. Hence
	\begin{align}\label{eq:finite element dimension lower bound}
		\dim\mathcal V_k(\mathcal T)
		\geqslant V
		\gtrsim_{d,\theta_0}N_{\mathcal T}.
	\end{align}
	Combining \eqref{eq:finite element dimension upper bound} and \eqref{eq:finite element dimension lower bound} gives $\dim\mathcal V_k(\mathcal T)
		\asymp_{d,k,\theta_0}N_{\mathcal T}$. Together with \eqref{eq:sparse storage lower bound} and the preceding storage upper bound, this proves \eqref{eq:dof optimal sparse scaling}.
\end{proof}

\begin{remark}[Sparse storage of the MLP]
The same scaling holds for the particular sparse MLP constructed by \eqref{eq:MLP parameter conversion}. Let $\widetilde P_{\mathcal T}$ count the retained entries of $\widetilde{\bm W}_\I$, $\widetilde{\bm b}_\I$, $\widetilde{\bm W}_\II$, $\widetilde{\bm b}_\II$, and $\widetilde{\bm W}_\III$, storing both copies of the paired first-layer parameters. Each of the two column blocks of $\widetilde{\bm W}_\II$ uses the union of the sparsity patterns of $\bm W_{\II,0}$ and $\bm W_{\II,1}$. Hence
\begin{align}
    \nslot(\widetilde{\bm W}_\I)+\nslot(\widetilde{\bm b}_\I)
    &=2\bigl(\nslot(\bm W_\I)+\nslot(\bm b_\I)\bigr),\\
    \nslot(\widetilde{\bm W}_\II)
    &\leqslant2\bigl(\nslot(\bm W_{\II,0})+\nslot(\bm W_{\II,1})\bigr).
\end{align}
The output matrix is unchanged, so
\begin{align}
    MN_\mathcal S\leqslant\widetilde P_{\mathcal T}
    \leqslant2P_{\mathcal T},
    \qquad
    \widetilde P_{\mathcal T}
    \asymp_{d,k,\theta_0}\dim\mathcal V_k(\mathcal T)
    \asymp_{d,k,\theta_0}N_{\mathcal T}.
\end{align}
In particular, the number of nonzero MLP parameters is $\mathcal O_{d,k}(N_{\mathcal T})$. The two-sided comparison concerns retained storage slots, as in Corollary~\ref{cor:dof optimality}; individual values may vanish for particular target functions.
\end{remark}

\subsection{Approximation on a uniform mesh}\label{sec:continuous approximation}

Corollary~\ref{cor:DG Sobolev approximation} gives estimates in broken Sobolev norms. We now construct continuous network approximants with global $W^{1,p}$ error bounds and, for $p=\infty$, uniform bounds on the closed domain. We use uniform Kuhn meshes to make the resulting network sizes explicit: their numbers of facet-supporting hyperplanes and elements grow as $\mathcal O_d(N)$ and $\mathcal O_d(N^d)$, respectively, leading to the same respective growth rates for the two hidden-layer widths.

Let $\Omega=(0,1)^d$, let $N$ be a positive integer, and set $h=1/N$. Divide the cube into $N^d$ cubes of side length $h$. On each small cube, use the same Kuhn triangulation: in local coordinates $\bm y\in(0,1)^d$, take the $d!$ simplices
\begin{align}
    0<y_{\pi(1)}<\cdots<y_{\pi(d)}<1,
\end{align}
where $\pi$ ranges over all permutations of $\{1,\ldots,d\}$. Denote the resulting conforming mesh by $\mathcal T_N$.

Its facet-supporting hyperplanes are
\begin{align}
    x_i&=j/N,
    &&1\leqslant i\leqslant d,\quad 0\leqslant j\leqslant N,\\
    x_i-x_j&=\ell/N,
    &&1\leqslant i<j\leqslant d,\quad -N+1\leqslant\ell\leqslant N-1.
\end{align}
The first family contains the grid hyperplanes. The second contains the images of the local ordering planes $y_i=y_j$; all the listed offsets occur. Hence
\begin{align}\label{eq:Kuhn counts}
    N_{\mathcal T_N}=d!N^d,\qquad
    H_{\mathcal T_N}
    =d(N+1)+\binom d2(2N-1).
\end{align}

These counts and Theorem~\ref{thm:continuous FE representation},
combined with conforming finite element approximation estimates,
yield continuous network approximants with global Sobolev error bounds.

\begin{corollary}[Continuous finite element approximation]
	\label{cor:continuous approximation}
	Let $k\geqslant1$, $1\leqslant p\leqslant\infty$, and let
	$1\leqslant r\leqslant k+1$ be an integer.
	Set $h=1/N$ on $\Omega=(0,1)^d$.
	For every $\bm v\in W^{r,p}(\Omega;\mathbb E)$,
	there exists a network
	\begin{align}\label{eq:continuous approximation class}
		\bm{\mathcal N}\in\fnn_d^m\left(
		k;\,2d^2N-d^2+3d,\,
		\frac{(d+k)!}{k!}N^d\right)
	\end{align}
	that is continuous on $\overline\Omega$ and satisfies
	\begin{align}\label{eq:continuous network approximation}
		\norm{\bm v-\bm{\mathcal N}}_{W^{s,p}(\Omega;\mathbb E)}
		\leqslant
		C h^{r-s}\abs{\bm v}_{W^{r,p}(\Omega;\mathbb E)},
		\qquad s=0,1.
	\end{align}
	The constant $C$ depends only on $d$, $k$, $r$, $p$, and $m$.
	When $p=\infty$, taking the continuous representative of
	$\bm v$ on $\overline\Omega$, the same network satisfies
	\begin{align}\label{eq:continuous uniform approximation}
		\sup_{\bm x\in\overline\Omega}
		\norm{\bm v(\bm x)-\bm{\mathcal N}(\bm x)}_{\mathbb E}
		\leqslant
		C h^r\abs{\bm v}_{W^{r,\infty}(\Omega;\mathbb E)}.
	\end{align}
\end{corollary}

\begin{proof}
	Apply the conforming finite element quasi-interpolation operator
	of \cite[Theorem~5.2]{ern2017quasi} to each component of $\bm v$.
	The local estimates, together with the uniformly bounded overlap
	of the element patches on $\mathcal T_N$, give
	$\bm v_N\in\mathcal V_k(\mathcal T_N;\mathbb E)$ satisfying
	\begin{align}
		\norm{\bm v-\bm v_N}_{L^p(\Omega;\mathbb E)}
		+h\abs{\bm v-\bm v_N}_{W^{1,p}(\Omega;\mathbb E)}
		\leqslant
		C h^r\abs{\bm v}_{W^{r,p}(\Omega;\mathbb E)}.
	\end{align}
	For integer $r$, these estimates hold for all
	$1\leqslant p\leqslant\infty$; at $p=\infty$, elementwise
	maxima replace summation.
	Since $h\leqslant1$, they also yield the full
	$W^{s,p}(\Omega;\mathbb E)$ bounds for $s=0,1$.
	
	The cube satisfies the star condition
	\eqref{eq:star shaped assumption}, and an interior point can be
	chosen away from the finitely many mesh hyperplanes.
	Theorem~\ref{thm:continuous FE representation} therefore gives
	a network $\bm{\mathcal N}$ that equals $\bm v_N$ at every point
	of $\overline\Omega$.
	The counts \eqref{eq:Kuhn counts} give the network class
	in \eqref{eq:continuous approximation class}.
	In particular, $\bm{\mathcal N}$ is continuous on
	$\overline\Omega$ and belongs to $W^{1,p}(\Omega;\mathbb E)$,
	so \eqref{eq:continuous network approximation} follows.
	
	For $p=\infty$, the continuous representatives of $\bm v$
	and $\bm v_N$ on $\overline\Omega$ satisfy
	\begin{align*}
		\sup_{\bm x\in\overline\Omega}
		\norm{\bm v(\bm x)-\bm v_N(\bm x)}_{\mathbb E}
		=
		\norm{\bm v-\bm v_N}_{L^\infty(\Omega;\mathbb E)}.
	\end{align*}
	Since $\bm{\mathcal N}=\bm v_N$ on $\overline\Omega$,
	\eqref{eq:continuous uniform approximation} follows.
\end{proof}

\begin{remark}[$L^p$ approximation with continuous activations]
	For $k\geqslant1$ and $1\leqslant p<\infty$, the discontinuous activation $\mathrm{ReLU}^0$ can be avoided if only $L^p$ approximation is required. Indeed, applying the localization construction of Jin~\cite{jin2026twohidden} to the positive affine functions $L_{\bm\beta}^{\tau}$ and using $\mathrm{ReLU}^k$ in the second hidden layer yields networks that agree with a given degree-$k$ finite element function outside arbitrarily thin neighborhoods of the mesh skeleton. For each fixed finite element function, these networks remain uniformly bounded as the neighborhoods shrink and hence converge to it in $L^p$. Their hidden-layer widths are independent of the neighborhood thickness, although their weights may grow. Thus the corresponding $L^p$ approximation rates can also be obtained using ReLU in the first hidden layer and $\mathrm{ReLU}^k$ in the second, with widths determined by this localization construction.
\end{remark}

\section{Conclusion}\label{sec:conclusion}

We have given a constructive correspondence between piecewise polynomial finite element functions and two-hidden-layer MLPs that combines a uniform architecture, linear sparse storage, and explicit parameter computation. The architecture applies in arbitrary dimension and polynomial degree on conforming simplicial meshes, with common hidden features for scalar, vector, and tensor outputs. Localizing positive affine powers gives exact elementwise polynomial reproduction, while the choice of geometric tests determines the values on the mesh skeleton. This yields both almost-everywhere DG representation and pointwise continuous representation on the closed domain.

For fixed dimension, degree, and output dimension, the sparse storage grows linearly with the number of elements. A half-open decomposition further reduces the continuous construction to one expansion per element; the star condition guarantees such a decomposition but is not needed for the general subsimplex construction. Under the shape regularity condition, the retained storage count for scalar continuous functions of degree at least one is comparable to the finite element dimension. All parameters are obtained from geometric formulas and local linear systems, whose reference factorizations can be reused. Exact realization transfers broken $W^{s,p}$ approximation estimates to DG network realizations and global $W^{1,p}$ estimates to continuous network realizations, without additional representation error.

The exact representation construction uses a discontinuous first-layer activation and a degree-dependent second-layer activation. Its explicit coefficient maps may also be ill-conditioned at higher degree. Further work includes numerically stable coefficient computation, quantitative control of weight growth, derivative errors in continuous approximations to the discontinuous gate, and smaller pointwise representations on general meshes. For compatible finite element spaces, it would also be useful to develop network constructions that explicitly encode the discrete differential operators and commuting relations in addition to representing the component functions.

\appendix
\section{Local polynomial coefficients}\label{app:coefficient computation}

This appendix gives the coefficient computation used in Section~\ref{sec:polynomial representation}. We use the same ordering $\iota$ as in the proof of Lemma~\ref{lem:k-linear decompose of Pk}. For $k=0$ and $p\in\P_0(\tau)$, the single coefficient is simply the constant value of $p$. For $k\geqslant1$, set
$\mathcal H_k^{d+1}=\{\widehat{\bm\alpha}\in\Nzero^{d+1}:|\widehat{\bm\alpha}|=k\}$
and define the bijection ${}^\sharp:\mathcal A_k^d\to\mathcal H_k^{d+1}$ by
$\bm\alpha^\sharp=(k-|\bm\alpha|,\alpha_1,\ldots,\alpha_d)$. Define the degree-$k$ Lagrange lattice
\begin{align}
	\mathcal X_k(\tau)
	=
	\left\{
	\bm x_{\widehat{\bm\alpha}}
	=
	\sum_{i=0}^d\frac{\widehat\alpha_i}{k}\bm x_i:
	\widehat{\bm\alpha}\in\mathcal H_k^{d+1}
	\right\}.
\end{align}
The node $\bm x_{\widehat{\bm\alpha}}\in\mathcal X_k(\tau)$ has barycentric coordinates $\widehat{\bm\alpha}/k$. If $p = \sum_{\bm\beta\in\mathcal A_k^d} u_{\bm\beta}(L_{\bm\beta}^\tau)^k$, then evaluation at the lattice node $\bm x_{\bm\alpha^\sharp}$ gives
\begin{align}
	p(\bm x_{\bm\alpha^\sharp})
	=
	\sum_{\bm\beta\in\mathcal A_k^d}
	u_{\bm\beta}
	\left(1+\frac1k\bm\alpha\cdot\bm\beta\right)^k.
\end{align}
Set $p_{\bm\alpha}=p(\bm x_{\bm\alpha^\sharp})$ and $A_{\bm\alpha\bm\beta}=\left(1+\frac1k\bm\alpha\cdot\bm\beta\right)^k$. Then $\bm p=\bm A\bm u$. The matrix $\bm V$ from the proof of Lemma~\ref{lem:k-linear decompose of Pk}, with $V_{\bm\gamma\bm\beta}=\bm\beta^{\bm\gamma}$, gives the factorization
\begin{align}
    \bm A=\bm V^{\mathrm T}\bm D_k\bm V,
    \qquad
    (D_k)_{\bm\gamma\bm\gamma}
    =\frac{k!}{(k-|\bm\gamma|)!\,\bm\gamma!\,k^{|\bm\gamma|}},
    \quad \bm\gamma\in\mathcal A_k^d.
\end{align}
This follows directly from the multinomial expansion of $A_{\bm\alpha\bm\beta}$. Since $\bm V$ is invertible and $\bm D_k$ is positive diagonal, $\bm A$ is symmetric positive definite. 

To obtain explicit triangular factors, define the Newton polynomials
\begin{align}
    \psi_{\bm\gamma}(\bm x)
    =\prod_{i=1}^d\prod_{j=0}^{\gamma_i-1}(x_i-j)
    =\sum_{\bm\alpha\in\mathcal A_k^d}
    L_{\bm\gamma\bm\alpha}\bm x^{\bm\alpha}.
\end{align}
Under the ordering $\iota$, their coefficient matrix $\bm L$ is unit lower triangular. Evaluation at the integer lattice gives
\begin{align}
    \bm U=\bm L\bm V,
    \qquad
    U_{\bm\gamma\bm\beta}
    =\psi_{\bm\gamma}(\bm\beta)
    =\begin{cases}
        \displaystyle\frac{\bm\beta!}{(\bm\beta-\bm\gamma)!},
        &\bm\gamma\leqslant\bm\beta,\\[1mm]
        0,&\text{otherwise},
    \end{cases}
\end{align}
where the inequality is componentwise. Thus $\bm U$ is upper triangular with diagonal entries $\bm\gamma!$, and $\bm V=\bm L^{-1}\bm U$. Substituting this identity into $\bm A=\bm V^{\mathrm T}\bm D_k\bm V$ reduces the computation of $\bm u=\bm A^{-1}\bm p$ to two triangular solves, two matrix-vector products and a diagonal scaling. All these matrices remain independent of the geometry of $\tau$, so they can be prepared once and reused across elements. For vector- and tensor-valued polynomials, the same factorization is applied to each component.

However, $\bm A$ may be ill-conditioned at higher polynomial degrees. For example, when $d=2$ and $k=6$, the dimension of $\bm A$ is $28$ and its Euclidean condition number is approximately $4.68\times10^9$. This may lead to numerical stability issues in the coefficient computation even when an explicit inverse is avoided.

\printbibliography

\end{document}